\documentclass[12pt]{article}
\usepackage{amsmath, latexsym, amsfonts, amssymb, amsthm, amscd}
\usepackage{xcolor}

\usepackage{hyperref}
\usepackage{xcolor}
\definecolor{link-color}{rgb}{0.15,0.4,0.15}
\hypersetup{
  colorlinks,
  linkcolor = {link-color}, citecolor = {link-color},
}

\newtheorem{theorem}{Theorem}

\newtheorem{lemma}{Lemma}

\newcommand{\p}{\Bbb{P}}

\newcommand{\e}{\Bbb{E}}

\DeclareMathOperator{\nn}{\boldsymbol{n}}

\newcommand{\ee}{\mbox{{\bf e}}}

\newcommand{\R}{\mathbb{R}}

\newcommand{\ud}{\mathrm{d}}

\title{
\textbf{The excursion measure away from zero for spectrally negative L\'evy processes.}}
\author{\textbf{ J.C. Pardo\footnote{Centro de Investigaci\'on en Matem\'aticas A.C. Calle Jalisco s/n. C.P. 36240, {\sc Guanajuato, Mexico.}  Email: jcpardo@cimat.mx.}\,\,, J.L. P\'erez\footnote{Department of Probability and Statistics, IIMAS, UNAM. , C.P. 04510 {\sc Mexico, D.F., Mexico.} Email: garmendia@sigma.iimas.unam.mx}\,\,, V.M. Rivero\footnote{Centro de Investigaci\'on en Matem\'aticas A.C. Calle Jalisco s/n. C.P. 36240, {\sc Guanajuato, Mexico.}  Email: rivero@cimat.mx.}\, .}}
\date{\footnotesize This version: \today}
\begin{document}

\maketitle

\begin{abstract}
We provide a description of the excursion measure from a point for a spectrally negative L{\'e}vy process. The description is based in two main ingredients. The first is building a spectrally negative L\'evy process conditioned to avoid zero and the study of its entrance law at zero. The latter is connected with both the excursion measure from zero of the process reflected in its infimum and reflected in its supremum. This lead us to establish a connection with the excursion measure from the state zero with the excursion measure from zero for the process reflected at the infimum and reflected at the supremum, respectively, which is the second main ingredient of our description.
\end{abstract}

\noindent{{\bf Keywords}: L\'evy processes, excursion theory from a point,  local times, fluctuation theory.}\\

\noindent{{\bf Mathematics subject classification}: 60G51, 60 G17

\section{Introduction}
Let $X=(X_{t}, t\geq 0)$ be a spectrally negative L\'evy process, viz a real valued stochastic process with non-monotone c\`adl\`ag paths, independent and stationary increments and no positive jumps. For $x\in \R$, {we} denote by $\p_x$ the law of $X$ when it is started at $x$ and write for convenience $\p$ in place of $\p_0$.  It is well known that in this case the Laplace exponent of $X,$ $\Psi:[0,\infty) \to \R$, {which is defined by}
\begin{equation*}\label{psi}
\e\Big[{\rm e}^{\lambda X_t}\Big]=:{\rm e}^{\Psi(\lambda)t}, \qquad t, \lambda\ge 0,
\end{equation*}
is {well-defined and satisfies the so called L\'evy--Khintchine formula, i.e. }

\begin{equation}\label{psi2}
\Psi(\lambda)=\gamma\lambda+\frac{\sigma^2}{2}\lambda^2+\int_{(-\infty,0)}\big({\rm e}^{\lambda x}-1-\lambda x\mathbf{1}_{\{x<1\}}\big)\Pi(\ud x),
\end{equation}
where $\gamma\in \R$, $\sigma^2\ge 0$ and $\Pi$ is a measure on $(-\infty,0).$ called the L\'evy measure of $X,$ {satisfying}
$$
\int_{(-\infty,0)}(1\land x^2)\Pi(\ud x)<\infty.
$$
We refer to Chapter VII in \cite{B} or Chapter VIII in \cite{K} for {further details on the theory of spectrally negative L\'evy processes or SNLP for short.} 

Our {aim in this paper is to develop It\^{o}'s excursion theory away from  $0$ for the process $X$.  With this purpose in mind,  we  assume} that the state $0$ is regular for itself and that $X$ is not a compound Poisson process. According to Corollary 5 in Chapter VII in \cite{B} this occurs if and only if $X$ has unbounded variation, which is known to be equivalent to either $\sigma^{2}>0$ or  
{\[
\int_{(-\infty,0)}(1\land |x|)\Pi(\ud x)=\infty.
\]}

General theory of Markov process, see e.g. Chapter IV in \cite{B}, ensures that there exists a local time at zero for $X,$ {here denoted by  $(L_{t}, t\geq0).$  The right-continuous inverse of the  local time  $L$ is defined as follows} 
\begin{equation}\label{ilt}
\tau_{t}=\inf\{s>0: L_{s}>t\},\qquad t\geq 0,
\end{equation}
{which is known to be a, possibly killed,  subordinator.} Let $\mathcal{E}$ the space of real valued c\'adl\'ag paths {with lifetime}, and $\Upsilon$ an isolated state. The main result in It\^{o}'s excursion theory, as described in~\cite{blumenthal-exc} or \cite{fitzsimmons-getoor-exc}, ensures that the point process $(\varepsilon_{t}, 0<t\leq L_{\infty})$ on $\mathcal{E}\cup\{\Upsilon\}$ defined by
$$\varepsilon_{t}=\begin{cases}
X_{\tau_{t-}+s}, 0\leq s\leq \tau_{t}-\tau_{t-}, & \text{if }  \tau_{t}-\tau_{t-}>0,\\
\Upsilon, & \tau_{t}-\tau_{t-}=0,
\end{cases} \qquad\textrm{for }\quad t\leq L_{\infty}$$
is a Poisson point process with intensity measure, say $\nn,$ stopped at the first point $\varepsilon\in \mathcal{E}$ whose lifetime is infinite. The measure $\nn$ is also called the excursion measure {away} from $0$ for $X.$ 

In recent years there have been several insightful papers reviewing diverse aspects of  excursion theory and its applications, several of them are dedicated to the memory of Professor K. It\^{o}. To name but a few we mention the paper by Watanabe~\cite{watanabe-exc} about diffusion processes, Le Gall's \cite{legall-exc} on {Brownian} motion and random trees, Werner's \cite{werner-exc} on scaling limits of critical systems in statistical physics, Pitman and Yor's \cite{pitman-yor-exc} and Yen and Yor's~\cite{yenyor} providing a general panorama of the excursion theory for Brownian motion, and Fitzsimmons and Getoor's \cite{fitzsimmons-getoor-exc} giving an updated review of the general theory. These papers and the references therein together with the book by Blumenthal~\cite{blumenthal-exc} and the papers by Getoor~\cite{getoor-exc} and Rogers~\cite{rogers-excII} should provide the necessary background and may serve as a guide for the reader interested in the applications of excursion theory.        

Our aim in this work is to provide a description {\it \`a la Imhof } of the measure $\nn$, which consists on first constructing the law of the process conditioned to avoid zero, then showing that this process can be started from zero, and finally removing the conditioning. The resulting measure is proportional to the excursion measure. We will actually achieve this program for the process killed at {an independent} exponential time of parameter $\beta$ and then we will make $\beta\to 0.$ These steps will be made precise below. This method is reminiscent of that introduced  by Imhof \cite{imhof} to describe $\nn$ in the case where $X$ is a Brownian motion, and has been successfully applied in many different contexts, see e.g. \cite{CD} for L\'evy processes reflected in its past infimum,  \cite{rivero-cramer}  for positive self-similar Markov processes, and \cite{yano-exc} for symmetrical real valued L\'evy processes. This study is motivated, first, by the popularity of spectrally negative L\'evy process in population dynamics, risk theory and financial mathematics, see e.g. \cite{K} where some of these topics are described, and second, from the fact that in the literature on L\'evy process only in the paper by Yano \cite{yano-exc} a systematic study of the excursion measure $\nn$ is carried, and this concerns mainly symmetric L\'evy processes with no-Gaussian part. Our study {complements that of Yano, in the sense} that we consider asymmetric L\'evy processes with no restriction on the Gaussian part. A further motivation for our study is the applications of some of the results here obtained to study a model for the reserve of an insurance company that falls into bankruptcy at the first time where the value of the reserve is below a random level depending only on the value of the reserve prior to the latest ruin, as described {by the authors} in \cite{PPR}.

The rest of this paper is organized as follows. In the Section~\ref{preliminaries} we introduce some further notation and recall some general facts {of} fluctuation theory of L\'evy processes. Then in Section~\ref{cond-exptime} we build SNLP conditioned to avoid zero up to an independent exponential time and describe the behaviour of this probability measure as the starting point tends to zero, so to obtain its entrance law, and then we enlarge the time interval where the conditioning applies. After performing this construction in Section \ref{excursiontheory}  we recall some general facts about { excursion theory away} from zero for $X$ and then we provide a description of the excursion measure in terms of the excursion measures of the process reflected in its infimum and supremum, respectively. {Our main results are stated at Sections~\ref{cond-exptime} and \ref{excursiontheory} and their proofs are in Sections~\ref{proofofthm12}, \ref{exc-calcul} and \ref{Proofthm3}.}

\section{{Preliminaries} on fluctuation theory}\label{preliminaries}

Let $\Omega$ be the space of right continuous functions with left limits from $(0,\infty)$ to $\R,$ endowed with the Skorohod topology. Let $\mathcal{F}$ stand for its Borel $\sigma$-algebra. For any path $\omega\in \Omega$, we define its lifetime $\zeta(\omega)$ by $$\zeta(\omega)=\inf\{t\geq 0: \omega(s)=\omega(t), \forall s\geq t\},\qquad \inf\{\emptyset\}=\infty.$$ For any $x\in \R$, we denote by $\tau^{+}_{x}$, respectively $\tau^{-}_{x},$  the first passage time above, respectively below, the level $x$ for $\omega,$ { i.e.}
$$\tau^{+}_{x}=\inf\{t>0: \omega(t)>x\},\qquad \tau^{-}_{x}=\inf\{t>0: \omega(t)<x\};$$ and by $T_{x}$ the first hitting time of $x,$ { i.e.} $$T_{x}=\inf\{t>0: \omega(t)=x\}.$$ We denote by $X$ the canonical process on $\Omega$. {In other words,}  $X_{t}(\omega)=\omega(t),$ $t\geq 0,$ and we consider  a probability measure $\p$ on $(\Omega, \mathcal{F}),$ under which the canonical process $X$ is a SNLP started at $0$  {with Laplace exponent $\Psi$ satisfying   (\ref{psi}). Recall that  for $x\in \R$,  we denote by $\p_x$ the law of $X+x$ under $\p$}. Accordingly, we shall write $\e_x$ and $\e$ for the associated expectation operators. 
By $\widehat{\p}$ we  denote the law of the dual process $\widehat{X}=(\widehat{X}_{t}:=-X_{t}, t\geq 0),$ under $\p.$ 

We {also} denote by $(\mathcal{F}_{t}, t\geq 0)$ the natural filtration of $X$ completed with the $\p$-null sets of $\mathcal{F}.$  Often we will consider continuous and bounded or measurable and positive functionals of the canonical process, so we introduce the classes, for $t\geq 0,$
$$\mathcal{H}_{b,t}=\{F:[0,t]\times\Omega\to \mathbb{R}, F\ \text{is } \mathcal{B}([0,t])\otimes\mathcal{F}_{t}\text{-measurable and bounded}\};$$ 
$$\mathcal{C}_{t}=\{F\in \mathcal{H}_{p,t}, F\ \text{is continuous w.r.t. Skorohod's topology} \};$$ 
$$\mathcal{H}(\mathbb{R})=\{f:\mathbb{R}\to \mathbb{R}, f \ \text{is measurable and positive}\},\  \mathcal{C}(\mathbb{R})=\{f\in \mathcal{H}(\mathbb{R}), f \ \text{is continuous}\}.$$  As usual we  denote for $t,s\geq 0$, by $\theta_{t}$ the shift, $\theta_{t}\omega(s)=\omega(t+s)$.

{In the sequel,  we will work under the assumption that $0$ is regular for $\{0\}.$ Recall that the latter } is equivalent to require that $X$ has unbounded variation, which in turn holds if and only if either
{\[
\sigma^{2}>0\qquad \textrm{or}\qquad  \int_{(-\infty,0)}(1\land |x|)\Pi(\ud x)=\infty.
 \]}
  A further differentiation between the cases where $\sigma^{2}>0$ and $\sigma^{2}=0$ will be necessary later.  It is important to {note  that $\sum_{s>0}\delta_{(s,\Delta X_{s})},$ the point measure of jumps of $X$ with $\Delta X_{s}=X_{s}-X_{s-},$ $s>0,$} is a Poisson point measure on $(0,\infty)\times\R\setminus\{0\}$ with intensity measure $\Lambda\otimes\Pi,$ {where $\Lambda$ denotes} the Lebesgue measure on $(0,\infty).$ The compensation formula for Poisson measures implies that  for any  predictable process $(G_{t}, t\geq 0)$ taking values in the space of non-negative measurable functions on $\mathbb{R},$ such that $G_{t}(0)=0,$ for all $t\geq 0,$ { the following identity holds} 
\begin{equation*}
\e\left(\sum_{0<s<\infty} G_{s}(\Delta X_{s})1_{\{\Delta X_{s}\neq 0\}}\right)=\e\left(\int^{\infty}_{0}\ud t\int_{\mathbb{R}\setminus\{0\}}\Pi(\ud y)G_{t}(y)\right),\end{equation*} see e.g. \cite{B} page 7.

\subsection{Fluctuation theory}\label{fluctutationtheory}

For each $q\geq0${,} we define the $q$-scale function as the unique function $W^{(q)}:
\R\to [0, \infty),$ such that $W^{(q)}(x)=0$ for all $x<0$ and on $(0,\infty)$ is continuous and with Laplace transform
\begin{equation*}
\int^{\infty}_0\mathrm{e}^{-\lambda x}W^{(q)}(x)dx=\frac1{\Psi(\lambda)-q},
\qquad \lambda>\Phi(q);\notag
\end{equation*}
where $ \Phi(q) = \sup\{\lambda \geq 0: \Psi(\lambda) = q\},$ which is well defined and finite for all $q\geq 0$, since $\Psi$ is a strictly convex function satisfying $\Psi(0) = 0$ and $\Psi(\infty) = \infty$. For convenience, we write $W$ instead of $W^{(0)}$. 

{We denote by 
\[
I_{t}=\inf_{s\leq t}\{X_{s}\wedge 0\} \qquad\textrm{and}\qquad S_{t}=\sup_{s\leq t}\{X_{s}\vee 0\},
\]  the past infimum and the past supremum, at time $t,$ for $t\geq 0.$  Recall {that} the process $X$ reflected at its past infimum  $(X_{t}-I_{t}, t\geq 0),$  respectively at its past supremum $(S_{t}-X_{t}, t\geq 0),$ is a strong Markov process with respect to the filtration $(\mathcal{F}_{t}, t\geq 0).$  As it is customary in  fluctuation theory of L\'evy processes we  denote by $\underline{\nn},$ and $\overline{\nn},$ respectively, the excursion measure of $X$ reflected at its past infimum, resp. at its past supremum.  In  recent years there has been an important research activity  describing the excursion measures $\underline{\nn}$ and $\overline{\nn},$ see  for instance \cite{CD}, \cite{doneybook}, \cite{duquesne} and the references therein. These measures will play a crucial rol in our description of $\nn$, the excursion measure away from $0$.  Similarly to Duquesne ~\cite{duquesne}, we  consider excursion measures, $\underline{\nn}$ and $\overline{\nn},$ that record the final position of the excursion $X_{\zeta},$ which describes the amount of the excursion overshoots when $X$ attains a new local extrema. }

{A crucial point in our arguments is the so-called spatio-temporal Wiener-Hopf factorisation. In particular, it says that  for any $q>0$, the Laplace exponent of $X$ can be written as follows}
\begin{equation*}
\frac{q}{q-\Psi(\lambda)}=\frac{\kappa(q,0)}{\kappa(q,-\lambda)}\frac{\widehat{\kappa}(q,0)}{\widehat{\kappa}(q,\lambda)}, \qquad \lambda>0,
\end{equation*}
where $\kappa$ and $\widehat{\kappa}$ are given by 
$$\kappa(\alpha,\beta)=\Phi(\alpha)+\beta,\qquad \widehat{\kappa}(\alpha,\beta)=\frac{\alpha-\Psi(\beta)}{\Phi(\alpha)-\beta},\qquad \alpha,\beta\geq 0,$$ and {represent the Laplace exponents} of two bivariate subordinators $((\sigma_{t}, H_{t}), t\geq 0)$ and $((\widehat{\sigma}_{t}, \widehat{H}_{t}), t\geq 0),$ (each coordinate is a subordinator), viz.
$$\kappa(q,\lambda)=-\frac{1}{t}\log\e\left(\exp\{-q\sigma_{t}-\lambda H_{t}\}\right),\qquad \lambda, q\geq 0.$$ 
{In particular, we observe that the following identity holds}
 $$q=\kappa(q,0)\widehat{\kappa}(q,0),\qquad q\geq 0.$$ The subordinator $H$, respectively $\widehat{H},$ has the same image as the supremum, respectively infimum, process of $X,$ and the subordinators $\sigma$ and $\widehat{\sigma}$ have the same image as the points of times where $X$ reaches its past supremum {and  past infimum, respectively}. Furthermore, {since the processes has no positive jumps} $((\sigma_{t}, H_{t}), 0\leq t\leq \zeta)$ has the same law as  $((\tau^{+}_{t}, t), 0\leq t\leq S_{\infty}).$

The Wiener-Hopf factorization, when {restricted} to time, implies that there are constants $a, \widehat{a}\geq 0$ such that for any $\beta\geq 0$ 
$$\Phi(\beta)=a\beta+\overline{\nn}(\mathbf{e}_{\beta}<\zeta),\qquad \Phi(0)=\underline{\nn}(\zeta=\infty);$$
$$\frac{\beta}{\Phi(\beta)}=\widehat{a}\beta+\underline{\nn}(\mathbf{e}_{\beta}<\zeta),\qquad \lim_{\beta\to 0}\frac{\beta}{\Phi(\beta)}=\underline{\nn}(\zeta=\infty),$$ and $a\widehat{a}=0.$ Since $X$ possesses paths of unbounded variation both $a=0$ and $\widehat{a}=0.$  {In the case when the  Wiener-Hopf factorization is restricted  to space, we have }
\begin{equation}\label{WHS}
\Psi(\lambda)=\left(\lambda-\Phi(0)\right)\widehat{\kappa}(\lambda),\qquad \lambda\geq 0,\end{equation}
where $\widehat{\kappa}(\lambda):=\widehat{\kappa}(0,\lambda)$, $\lambda \geq 0,$ { denotes the Laplace exponent of  $\widehat{H}$, also known as the downward ladder height subordinator.} The L\'evy-Khintchine description for $\widehat{\kappa}(\cdot)$ is given by
$$\widehat{\kappa}(\lambda)=\Psi^{\prime}(0+)\vee 0+\frac{\sigma^{2}}{2}\lambda+\int^{\infty}_{0}\Pi_{\widehat{H}}(\ud x)(1-e^{-\lambda x}),$$ 
with $$\overline{\Pi}_{\widehat{H}}(x):=\overline{\Pi}_{H}(x,\infty)=e^{\Phi(0)x}\int^{\infty}_{x}\ud z e^{-\Phi(0)z}\overline{\Pi}^{-}(z), \qquad x>0;$$ 
 and $\overline{\Pi}^{-}(x)=\Pi(-\infty,-x),$ for $x>0.$
{The reader is referred to the monographs of Bertoin \cite{B} and Kyprianou \cite{K} for a complete introduction to  fluctuation theory of L\'evy processes, and in particular to Chapter VII in \cite{B} and Chapter VIII in \cite{K} for the specific case of spectrally negative L\'evy processes.}\\

For $\beta>0$, let $g^{+}_{\beta}, g^{-}_{\beta}:\mathbb{R}\to\mathbb{R}$ be the functions defined by
\begin{equation*}\label{161}
g^{+}_{\beta}(x)=\frac{\p_{x}(\tau^{-}_{0}>\mathbf{e}_{\beta})}{\underline{\nn}\left(\zeta>\mathbf{e}_{\beta}\right)},\qquad g^{-}_{\beta}(x)=\frac{\p_{x}(\tau^{+}_{0}>\mathbf{e}_{\beta})}{\overline{\nn}\left(\zeta>\mathbf{e}_{\beta}\right)}=\frac{1-e^{x\Phi(\beta)}}{\Phi(\beta)},\qquad x\in\mathbb{R}. 
\end{equation*} 
The results by Chaumont and Doney \cite{CD} imply that the function $g^{+}_{\beta},$ respectively $g^{-}_{\beta},$ are excessive for the process $X$ killed at its first passage time below, resp. above, zero. Moreover, they proved that $g^{+}_{\beta}$, resp. $g^{-}_{\beta},$ converges, as $\beta\to 0,$ towards an excessive function for $X$ killed at its first passage time below, resp. above, $0,$ which is in fact invariant when $X$ does not drift towards $-\infty,$ resp. towards $+\infty.$ Moreover, this convergence is monotone increasing. Hereafter, we will denote these limits by $g^{+},$ and $g^{-},$ respectively. Furthermore, the fact that $X$ is a SNLP implies that
\begin{equation*} 
g^{-}(x)=\frac{1-e^{\Phi(0)x}}{\Phi(0)},\qquad x<0,
\end{equation*}
and the right most term is interpreted in the limiting sense when $\Phi(0)=0,$ equivalently $\Psi^{\prime}(0+)\geq 0,$ and it is hence equal to $|x|.$ Hereafter we will denote by $g$ the function $$g(x)=\frac{1-e^{\Phi(0)x}}{\Phi(0)},\qquad x\in \mathbb{R},$$ where as above, the quotient is understood in the limit sense when $\Phi(0)=0.$  We also have  
\begin{equation*}
g^{+}(x)=W(x),\qquad x\geq 0.
\end{equation*}
Chaumont and Doney~\cite{CD} showed that in general the law of a L\'evy process conditioned to stay positive or to stay negative,  {here denoted by }$\e^{\uparrow}$ and  $\e^{\downarrow}$ respectively, are obtained via {a Doob $h$-transform, i.e.}
\begin{equation}\label{eq:cond+}
\p^{\uparrow}_{x}|_{\mathcal{F}_{t}\cap\{t<\zeta\}}=\frac{g^{+}(X_{t})}{g^{+}(x)}\mathbf{1}_{\{t<\tau^{-}_{0}\}}\p_{x}|_{\mathcal{F}_{t}},\qquad x>0,
\end{equation}
 and  
\begin{equation}\label{eq:cond-}
\p^{\downarrow}_{x}|_{\mathcal{F}_{t}\cap\{t<\zeta\}}=\frac{g^{-}(X_{t})}{g^{-}(x)}\mathbf{1}_{\{t<\tau^{+}_{0}\}}\p_{x}|_{\mathcal{F}_{t}},\qquad x<0.
\end{equation}  
{One of the main results in \cite{CD} establishes that the excursion measures of $X$ reflected in its infimum, $\underline{\nn}$, and  reflected in its supremum, $\overline{\nn},$ can be constructed as  limits of the law of $X$ killed at its first passage time below  and  above $0, $ respectively.} Namely,  there are constants $c_{+}\in(0,\infty)$ and $c_{-}\in(0,\infty),$ such that for any $t>0,$ and $F\in \mathcal{C}_{t},$
{\begin{equation*}\label{ninf}
\lim_{x\to 0+}\frac{1}{g^{+}(x)}\e_{x}\left(F, t<\tau^{-}_{0}\right)=c_{+}\underline{\nn}\left(F, t<\zeta\right)\quad\textrm{and}\quad
\lim_{x\to 0-}\frac{1}{g^{-}(x)}\e_{x}\left(F, t<\tau^{+}_{0}\right)=c_{-}\widehat{\overline{\nn}}\left(F, t<\zeta\right).
\end{equation*}}    
where $\widehat{\overline{\nn}}$ denotes the push forward measure of $\overline{\nn}$ under the function that maps a path $w\in\Omega$ into its negative $-w.$  {Furthermore, we have
\[
\lim_{x\to 0+}\e^{\uparrow}_{x}\left(F, t<\zeta\right)=c_{+}\underline{\nn}(Fg^+(X_t),t<\zeta)\quad\textrm{ and }\quad \lim_{x\to 0-}\e^{\downarrow}_{x}\left(F, t<\zeta\right)=c_{-}\widehat{\overline{\nn}}(Fg^-(X_t),t<\zeta).\notag 
\]}
The constants $c_{+}$ and $c_{-}$ depend on the {normalization} of the local time at zero for the process reflected in the infimum and supremum, respectively, and so hereafter we will assume that these have been {normalized i.e. $c_{+}=1=c_{-}.$}

We recall that the $q$-resolvent of $X$ {is defined as follows} $$U_{q}(\ud y):=\int^{\infty}_{0}e^{-qt}\p(X_{t}\in \ud y)\ud t ,\qquad y\in \mathbb{R},$$
is absolutely continuous with respect to Lebesgue measure and its density is given by 
\begin{equation}\label{resolvent-density}
u_{q}(y)=\Phi^{\prime}(q)e^{-\Phi(q)y}-W^{(q)}(-y),\qquad y\in \mathbb{R},
\end{equation} 
see for instance Corollary 8.9 in \cite{K} and Exercise 2 in Chapter VII in \cite{B}. 
Corollary 18 in Chapter II in \cite{B} imply that
\begin{equation}\label{eq:fhtlt}
\p_{y}(\ee_{\beta}<T_{0})=1-\e_{y}(e^{-\beta T_{0}})=1-\frac{u_{\beta}(-y)}{u_{\beta}(0)},\quad  y\in \mathbb{R}.\end{equation} 
The following identity will be needed later, for any $f:\mathbb{R}^{2}\to\mathbb{R}^{+}$ measurable
\begin{equation*}\label{32-1}
\begin{split}
K_{\beta}f(x):=&\e_{x}\left(f(X_{\tau^{-}_0-}, X_{\tau^{-}_0})1_{\{X_{\tau^{-}_{0}-}>0\}},\  \tau^{-}_0<\ee_{\beta}\right)\\
&=\int^{\infty}_{0}\ud y\left(e^{-\Phi(\beta)y}W^{(\beta)}(x)-W^{(\beta)}(x-y)\right)\int_{(-\infty,-y)}\Pi(\ud z)f(y,y+z),
\end{split}
\end{equation*}
which can be deduced from Theorem 2.7 in page 123 and equation (48) in page 125 in \cite{KKR}. The above expression defines the kernel
\begin{equation}\label{KB}
\begin{split}
K_{\beta}(x,\ud y, \ud z):=\left(e^{-\Phi(\beta)y}W^{(\beta)}(x)-W^{(\beta)}(x-y)\right)\Pi(\ud z-y)\ud y\mathbf{1}_{\{y>0, z<0\}},\qquad x>0.
\end{split}
\end{equation}{When $\beta=0$, we  simply denote this kernel by $K$.}
A similar identity also holds under {the measure} $\underline{\nn}.$ Indeed, using that under $\underline{\nn}$ the canonical process has the same transition probabilities as $X$ killed at its first passage time below $0$, it is easily verified that for any $t>0,$ $F\in\mathcal{H}_{b,t}$
\begin{equation}\label{eq:b9}
\begin{split}
\underline{\nn}(F f(X_{\zeta-}, X_{\zeta}), t<\zeta<\ee_{\beta})=
\underline{\nn}(FK_{\beta}f(X_{t}), t<\zeta\wedge \ee_{\beta}).
\end{split}
\end{equation}

\section{Conditioning to avoid $0$ up to an exponential time}\label{cond-exptime}
In \cite{henryavoid}, Pant\'\i \ gave a construction of L\'evy processes conditioned to avoid zero under some {technical conditions}, none of them being exclusive of spectrally one sided processes, but rather some integrability properties on the characteristic exponent are required. {The first step in our construction is rather similar to that in  \cite{henryavoid},  in other words we start by constructing the process conditioned to stay positive up to an independent exponential time of parameter $\beta,$ here denoted by $\mathbf{e}_\beta$. Then,  in \cite{henryavoid},  the parameter $\beta$ is sent to $0$ and the existence of a non-degenerate limit is established. Here we make a more detailed study by first showing that the process conditioned to avoid zero can be started from the origin and then  the limit law is studied as the parameter $\beta\to 0+.$ We will see that in our particular  case, i.e. SNLP, the assumptions in \cite{henryavoid} are not relevant.}\\

Let $\e^{\beta}_{x}$ denote the law of {the process $X$ issued from $x$ and killed at $\ee_{\beta},$} and $h_{\beta}$ be the function defined by 
\begin{equation}\label{eq:111}
h_{\beta}(y):=u_{\beta}(0)-u_{\beta}(-y)=\Phi^{\prime}(\beta)\left(1-e^{y\Phi(\beta)}\right)+\left(W^{(\beta)}(y)-W^{(\beta)}(0)\right), \qquad y\in \mathbb{R}.
\end{equation}
{Since $X$ has unbounded variation paths, we have $W(0)=0.$} This function plays a crucial role in our construction because it is an excessive function for the process killed at { $\ee_{\beta},$}} and in fact it is related to conditioning to avoid $0,$  see~\cite{henryavoid}. The reason for the latter  lies in (\ref{eq:fhtlt}), which implies the equality
\begin{equation}\label{hb}
h_{\beta}(y)=u_{\beta}(0)\p_{y}(\ee_{\beta}<T_{0})=u_{\beta}(0)\p^{\beta}_{y}(T_{0}=\infty), \qquad y\in \mathbb{R}.
\end{equation}

For $\beta\geq 0,$ let $\e^{\beta,0}$ denote the law of $X$ under $\e^{\beta}$ killed at its first hitting time of zero. { For simplicity when $\beta=0$, we  replace $\e^{0,0}$ by $\e^{0}.$ Observe that this is equivalent to {killing} $X$ under the law $\e$ at  time $T_{0}\wedge \ee_{\beta},$ and the reason for this is that the action of  killing at $T_0$, the process killed at an exponential time, is equivalent to killing the original process at whichever come first, the ringing of the exponential clock or the first hitting time of $0.$} In particular the semi-group of $X$ under $\e^{\beta,0},$ which will be denoted by $({P}^{\beta,0}_{t}, t\geq 0),$ has the {following }form $${P}^{\beta,0}_{t}(x,\ud y):=\p^{\beta}_{x}(X_{t}\in \ud y, t<T_{0})=\p_{x}(X_{t}\in \ud y, t<T_{0}\wedge \mathbf{e}_\beta\ )=e^{-\beta t}\p_{x}(X_{t}\in \ud y, t<T_{0}),$$ for $x,y\in \mathbb{R},\ t\geq 0.$  We  denote by 
$(\p^{\updownarrow,\beta}_{x}, x\in \mathbb{R}\setminus\{0\})$ the law $(\p^{\beta,0},  x\in \mathbb{R}\setminus\{0\})$ conditioned on the event $\{T_{0}=\infty\},$ which has a strictly positive probability since $$\p^{\beta,0}_{x}(T_{0}=\infty)=\p_{x}(T_{0}>\ee_{\beta})>0,\qquad x\in \mathbb{R}\setminus\{0\}.$$ {Thanks to  the Markov property, we observe that}  the law $\p^{\updownarrow,\beta}$ coincides with Doob's $h$-transform of $\e^{\beta,0}$ associated to the excessive function $h_{\beta}$, which is characterized by the relation  
$$\e^{\updownarrow,\beta}_{x}\left(\mathbf{1}_{\{F\cap\{t<\zeta\}\}}\right)=\frac{1}{h_{\beta}(x)}\e^{\beta,0}_{x}\left(\mathbf{1}_{\{F\cap\{t<\zeta\}\}}h_{\beta}(X_{t})\right),\qquad F\in\mathcal{F}_{t}, t\geq 0,\  x\in\{z\in \mathbb{R}, h_{\beta}(z)\neq 0\}.$$  

{A} natural question is whether the process conditioned to avoid zero up to an exponential time can be started from $0.$ If we take into account the case of a L\'evy process conditioned to stay positive, we can realize {that} the answer should be positive in general. Not only that, the knowledge of a SNLP conditioned to stay positive gives some intuition of a way to construct such process. Namely, one should first chose the sign of the path immediately after leaving zero, if the the sign is negative the path should keep that sign forever and so it should be a process conditioned to stay negative and started from zero; if the path starts positive it should either remain positive forever, and hence it should be a process conditioned to stay positive and started from $0,$ or else stay positive for a strictly positive amount of time and then jump to a level strictly negative and remain negative forever, which is a behaviour that can be reproduced by concatenating an excursion of the process reflected in the infimum with a process conditioned to stay negative. This also suggest that we need to differentiate on whether the process starts from zero {from below or from above} and hence consider two initial states $0+$ and $0-.$ Then a suitable  state space for the process conditioned to stay positive is $E_{0\pm}=(-\infty,0-]\cup[0+,\infty),$ the topological direct sum of $(-\infty,0]$ and $[0,\infty).$

We will show below that the { construction  described above} can be performed formally by letting $x\to 0+$ and $x\to 0-$. An evidence that this project is amenable can be found in \cite{millar} where the following Lemma is proved, and from it one can actually derive a finite dimensional convergence. 
\begin{lemma}[\cite{millar}]
Let 
\[
\displaystyle U^{\updownarrow,\beta}_{\lambda}(x,\ud y):=\e^{\updownarrow,\beta}_{x}\left(\int^{\zeta}_{0}e^{-\lambda t}\mathbf{1}_{\{X_{t}\in \ud y\}}\ud t\right),
\]
 denote the $\lambda$-resolvent of the canonical process under the law $\e^{\updownarrow,\beta}.$ We have
\begin{itemize}
\item $\displaystyle \frac{U^{\updownarrow,\beta}_{\lambda}(x,\ud y)}{\ud y}=\left(u_{\beta+\lambda}(y-x)-\frac{u_{\beta+\lambda}(-x)}{u_{\beta+\lambda}(0)}u_{\beta+\lambda}(y)\right)\frac{h_{\beta}(y)}{h_{\beta}(x)},\qquad y\neq 0, x\neq 0;$
\item as $x\uparrow 0$ (respectively, as $x\downarrow 0$) the measure $U^{\updownarrow,\beta}_{\lambda}(x,\ud y)$ converges weakly towards a finite measure $U^{\updownarrow,\beta}_{\lambda}(0-,\ud y)$, (respectively, $U^{\updownarrow,1}_{\beta}(0+,\ud y)).$
\end{itemize}
\end{lemma}
This is essentially the Lemma~4.4 in \cite{millar},  where the proof is given for $\beta=1,$ but a perusal of Millar's arguments allows to ensure that the parameter $1$ in the exponential time has nothing special and then it can be replaced {with} any $\beta>0.$ We will not develop the idea of establishing from this Lemma a finite dimensional convergence because we will establish a stronger type of convergence. This is the purpose of the following theorem.

\begin{theorem}\label{teo:1}
Let $X$ be a SNLP {with unbounded variation paths}. For $\beta>0,$ there are non-degenerate probability measures $\p^{\updownarrow,\beta}_{0-}$ and $\p^{\updownarrow,\beta}_{0+},$    which are the weak {limits} of the measures $\p^{\updownarrow,\beta}_{x}$ either as $x\downarrow 0$ or $x\uparrow 0,$ respectively, and they are characterised as follows,
\begin{itemize}
\item[(i)] for any $t>0$ and $F\in\mathcal{C}_{t}$,
$$\e^{\updownarrow,\beta}_{0-}(F, t<\zeta):=\lim_{x\uparrow 0}\e^{\updownarrow,\beta}_{x}(F, t<\zeta)= \frac{1}{\Phi(\beta)\Phi^{\prime}(\beta)}\widehat{\overline{\nn}}(F h_{\beta}(X_{t}), t<\zeta\wedge\ee_{\beta});$$
\item[(ii)] for any $t>0$, $s\geq 0,$ $F\in\mathcal{C}_{t}$, $G\in\mathcal{C}_{s},$ and $f:\mathbb{R}^{2}\to\mathbb{R}$ continuous and bounded,
\begin{equation}\label{th1:02}
\begin{split}
&\e^{\updownarrow,\beta}_{0+}(F f(X_{\tau^{-}_0-}, X_{\tau^{-}_0} )G\circ\theta_{\tau^{-}_{0}}, t<\tau^{-}_0<\tau^{-}_0+s< \zeta)\\
&\hspace{1cm}=\lim_{x\downarrow 0}\e^{\updownarrow,\beta}_{x}(F f(X_{\tau^{-}_0-}, X_{\tau^{-}_0} )G\circ\theta_{\tau^{-}_{0}}, t<\tau^{-}_0\leq\tau^{-}_0+s< \zeta)\\
&\hspace{1cm}=
\frac{1}{1-\frac{\sigma^{2}}{2}\Phi^{\prime}(\beta)\Phi(\beta)}\underline{\nn}(Fh_{\beta}(X_{\zeta})f(X_{\zeta-}, X_{\zeta} )\mathbf{1}_{\{X_{\zeta-}>0\}}\e^{\updownarrow,\beta}_{X_{\zeta}}\left(G, s<\zeta\right), t<\zeta< \ee_{\beta});\\
\end{split}
\end{equation}
\item[(iii)] for any $t>0$ and $F\in\mathcal{C}_{t}$ 
{\begin{equation*}
\begin{split}
\e^{\updownarrow,\beta}_{0+}\left(F, t<\zeta, \tau^{-}_0=\infty\right)&=\lim_{x\downarrow 0}\e^{\updownarrow,\beta}_{x}\left(F, t<\zeta, \tau^{-}_0=\infty\right)=\frac{\Phi^{\prime}(\beta)}{1-\frac{\sigma^{2}}{2}\Phi(\beta)\Phi^{\prime}(\beta)}{\underline{\nn}}(F, t<\ee_{\beta}<\zeta).
\end{split}
\end{equation*}}
\end{itemize}
\end{theorem}
So far we have been able to build a SNLP conditioned to stay positive with state space $E_{0\pm},$ up to an a.s finite time. The final question is whether it is possible to define this process over an interval of infinite length. For that end, we make now $\beta\to 0.$ 

\begin{theorem}\label{theo:2}
Let $X$ be a SNLP {with unbounded variation paths}. We have the following convergences 
\begin{itemize}
\item[(i)] for any $t>0$, and for any $F\in \mathcal{C}_t$ 
\begin{equation*}
\lim_{\beta\downarrow0}\e^{\updownarrow,\beta}_{0-}(F,t<\zeta)=
\displaystyle \widehat{\overline{\nn}}(Fg^{-}(X_{t}); t<\zeta); 
\end{equation*}
\item[(ii)] assume $\Pi\not\equiv 0,$ if either $X$ drifts towards $-\infty$ or does not drift to $-\infty$ and the variance is finite, then for any $t>0,$ $s\geq 0,$ $F\in \mathcal{C}_{t}$, $G\in \mathcal{C}_{s}$ and $f:\mathbb{R}^{2}\to\mathbb{R}$ continuous and bounded function we have
\begin{equation}\label{condlawexc}
\begin{split}
&\lim_{\beta\to 0}\e^{\updownarrow,\beta}_{0+}(F f(X_{\tau^{-}_0-}, X_{\tau^{-}_0} )G\circ\theta_{\tau^{-}_{0}}, t<\tau^{-}_0<\tau^{-}_0+s<\zeta | \tau^{-}_{0}<\zeta)\\
&\hspace{3cm}=\frac{\underline{\nn}\left(Fg^{-}(X_{\zeta})f(X_{\zeta-}, X_{\zeta} )1_{\{X_{\zeta-}>0\}}{\e^{\downarrow}_{X_{\zeta}}}\left(G, s<\zeta\right), t<\zeta<\infty\right)}{\underline{\nn}\left(g^{-}(X_{\zeta}), \zeta<\infty\right)},
\end{split}
\end{equation}with $$\underline{\nn}\left(g^{-}(X_{\zeta}), \zeta<\infty\right)=\int^{\infty}_{0}e^{\Phi(0)y}y\overline{\Pi}^{-}(y)\ud y<\infty;$$ while if $X$ does not drift to $-\infty$ and the variance is infinite the above limit is equal to zero.
\item[(iii)] for any $t>0$ and $F\in \mathcal{C}_{t}$ 
\begin{equation*}
\begin{split}
&\lim_{\beta\downarrow0}\e^{\updownarrow,\beta}_{0+}(F, t<\zeta | \tau^{-}_0=\infty)={\underline{\nn}}\left(Fg^{+}(X_{t}), t<\zeta\right)=\e^{\uparrow}(F\mathbf{1}_{\{t<\zeta\}}).
\end{split}
\end{equation*}
\end{itemize}
\end{theorem}

\section{Excursion theory from $0$ for $X$ }\label{excursiontheory} 
{Since  we assume that  $0$ is a regular point, there exists a local time at $0,$ i.e.   a continuous additive functional of $X$ that grows only on the instants where $X$ visits $0.$} By the formula~(2.4) and (2.6) in \cite{fitzsimmons-getoor-exc}, we have that the local time is related to $u_{\cdot}(0)$ by the formulas
$$u_{\lambda}(0)=\e\left(\int^{\infty}_{0}e^{-\lambda t}\ud L_{t}\right),\qquad \e\left(\exp\{-\lambda \tau_{t}\}\right)=\exp\left\{{-\frac{t}{u_{\lambda}(0)}}\right\},\qquad \lambda \geq 0,$$ with $\tau$ as defined in (\ref{ilt}). This implies that the Laplace exponent of the subordinator $\tau$ is given by $1/\Phi^{\prime}(\lambda), \lambda\geq 0.$ Moreover, since the local time is the unique, up to multiplicative constants, that growths only at the instants where the process $X$ visits $0$, it follows that there exists a constant $\delta\geq 0$ such that
$$\delta L_{t}=\int^{t}_{0}1_{\{X_{s}=0\}}\ud s,\qquad t \geq 0.$$ The constant $\delta$ coincides with the drift of the subordinator $\tau,$ and hence by Proposition 2-(ii) in Chapter I in~\cite{B},  we have that
$$\delta=\lim_{\lambda \to\infty}\frac{1}{\lambda\Phi^{\prime}(\lambda)}=0,$$ and the final equality {follows} from the fact that $\Phi$ is the inverse of $\Psi$ and (\ref{psi2}).
 
The master formula in excursion theory guarantees that if $\mathcal{G}$ denotes the set of left extrema of the excursion intervals then  
\begin{equation*}
\e\left(\sum_{s\in \mathcal{G}}Z_{s}F\circ\theta_{s}\right)=\e\left(\int^{\infty}_{0}{\rm d}L_{s}Z_{s}\right)\nn(F),
\end{equation*}
where {$\theta$} denotes the usual shift, $Z$ is any positive predictable process, and $F$ any positive measurable functional. This formula is one of the main tools for performing calculations using excursion theory. One consequence of this is  $$u_{\lambda}(0)=\frac{1}{\nn(1-e^{-\lambda \zeta})}, \qquad \lambda\geq 0.$$ A further consequence is the first exit decomposition formula, it states that for any $f\in \mathcal{H}$  
\begin{equation*}\label{firstexit}
\e_{x}\left(f(X_{t})\right)=\e_{x}\left(f(X_{t})\mathbf{1}_{\{t<T_{0}\}}\right)+\e_{x}\left(\int^{t}_{0}\ud L_{s}\nn(f(X_{t-s}), t-s<\zeta)\right),\qquad t\geq 0,
\end{equation*}
{see for instance formula (7.19) in \cite{getoor-exc} for its proof.} Essentially by integrating this with respect to $e^{-qt}\ud t$ on $(0,\infty),$ and applying the strong Markov property at $T_{0}$, we obtain the classical formula
\begin{equation*}\label{resolvent1}
U_{q}f(x)=\e_{x}\left(\int^{T_{0}}_{0}e^{-qt}f(X_{t})\ud t\right)+\e_{x}\left(e^{-q T_{0}}\right)\frac{\nn\left(\displaystyle\int^{\zeta}_{0}e^{-qt}f(X_{t})\ud t\right)}{\nn(1-e^{-q\zeta})}.
\end{equation*}{ It is also known that, under $\nn$, the canonical process is still  Markovian with the same semigroup as $X$ killed at $T_{0},$ i.e.}
$$\nn\left(F(X_{s}, s\leq t)f(X_{t+s}), t+s<\zeta\right)=\nn\Big(F(X_{s}, s\leq t)\e_{X_{t}}\left(f(X_{s}), s<T_{0}\right), t<\zeta\Big),\quad t>0, s\geq 0,$$ for any $F\in\mathcal{H}_{t}$ and $f\in \mathcal{H}.$ An useful consequence of the Markov property under $\nn$ and the fact that $\sum_{s>0}\delta_{(s,\Delta X_{s})},$ is a Poisson point measure on $(0,\infty)\times\R\setminus\{0\}$ with intensity measure $\Lambda\otimes\Pi,$ is that the compensation formula still holds under $\nn$. Namely, for any  predictable process $G$ taking values in the space of non-negative measurable functions on $\mathbb{R},$ such that $G_{t}(0)=0,$ for all $t\geq 0,$ we have the identity 
\begin{equation}\label{compensation}
\nn\left(\sum_{0<s<\infty} G_{s}(\Delta X_{s})1_{\{\Delta X_{s}\neq 0\}}\right)=\nn\left(\int^{\infty}_{0}\ud t\int_{\mathbb{R}\setminus\{0\}}\Pi(\ud y)G_{t}(y)\right).\end{equation}
 {The previous identity follows easily from  the fact that under $\nn$ the coordinate process is  Markovian with the same transition semigroup as that of $X$ killed at $T_{0}.$ Notice that this fact actually holds for any L\'evy process for which $0$ a is regular state.} 

{A} key in understanding $\nn$ is its associated entrance law $(\nn_{t}, t\geq 0),$ defined by $$\nn_{t}f:=\nn(f(X_{t}), t<\zeta),\qquad t>0,$$ for $f\in\mathcal{H}$. A consequence of Theorem 3.6 in \cite{fitzsimmons-getoor-exc} is that there is a constant $c>0$ such that
\begin{equation}\label{eq:trrrr}
\begin{split}
\int^{\infty}_{0} e^{-qt} \nn_{t}f\ud t&=c\int_{\R}f(x)\widehat{\e}_{x}(e^{-q T_{0}})\ud x \\
&=c\int_{\R} f(x)\frac{u_{q}(x)}{u_{q}(0)}\ud x=c\int_{\R}f(x)\left(e^{-\Phi(q)x}-\frac{1}{\Phi^{\prime}(q)}{W^{(q)}(-x)}\right)\ud x.
\end{split}
\end{equation}
This fact implies the {following} useful identity  
\begin{equation}\label{potentialkilled}
\begin{split}
&\nn\left(\int^{\zeta}_{0} e^{-\lambda s}f(X_{s})\mathbf{1}_{\{\tau^{-}_0>s\}}\ud s\right)=c\int^{\infty}_{0} e^{-\Phi(\lambda)x}f(x)\ud x, 
\end{split}
\end{equation}  
which holds for any $\lambda>0$ and any $f\in\mathcal{H}.$  {The latter follows by observing }
\begin{equation*}
\begin{split}
&\nn\left(\int^{\zeta}_{0} e^{-\lambda s}f(X_{s})\mathbf{1}_{\{\tau^{-}_0>s\}}\ud s)\right)=\nn\left(\int^{\zeta}_{0} e^{-\lambda s}f(X_{s})\mathbf{1}_{\{X_s>0\}}\ud s\right)=\int^{\infty}_{0}f(x) \frac{u^{\lambda}(x)}{u^{\lambda}(0)} \ud x,
\end{split}
\end{equation*}
{where the first identity is  consequence of the fact that  under the excursion measure the  events $\{\tau^{-}_{0}>s\}$ and $\{X_{s}>0\}$ coincide, since the process does not have positive jumps.} Another consequence of the identity (\ref{eq:trrrr}) is that there is a bi-measurable function $\rho(\cdot,\cdot)$ such that 
\begin{equation*}\label{eq:positiveentrancelaw}
c^{-1}\frac{\nn(X_{t}\in \ud x, t<\zeta)}{\ud x}=\rho(t,x)=\frac{\widehat{\p}_{x}(T_{0}\in \ud t, T_{0}<\infty)}{\ud t}, \qquad x\neq0, t>0.\end{equation*} See Theorem 3.6 in \cite{yano-exc} or Proposition 10.10 in \cite{getoor-exc} for  {a} proof of this fact. On the other hand, the absence of positive jumps implies 
\begin{equation*}
\rho(t,x)\ud t=\widehat{\p}_{x}(T_{0}\in \ud t, T_{0}<\infty)=\p(\tau^{+}_{x}\in \ud t, \tau^{+}_{x}<\infty), \qquad x>0.
\end{equation*} But  identity (8) in page 203 in \cite{B} ensures the existence of a constant $k^{\prime}$ such that
\begin{equation*}\label{eq:positiveentrancelaw2}
\rho(t,x)\ud t\ud x=\p(\tau^{+}_{x}\in \ud t, \tau^{+}_{x}<\infty)\ud x=k^{\prime}\overline{\nn}\left(X_{t}\in \ud x, t<\zeta\right)\ud t, \qquad x\neq0, t>0.
\end{equation*}
Putting all the pieces together, we deduce the equality of measures on $(0,\infty)\times(0,\infty)$
\begin{equation*}\label{eq:positiveentrancelaw3}
{\nn(X_{t}\in \ud x, t<\zeta)}\ud t=k^{\prime}\overline{\nn}\left(X_{t}\in \ud x, t<\zeta\right)\ud t.
\end{equation*}
This identity suggest that we should expect the equality of measures  
$$\nn(X_{t}\in \ud x, t<\tau^{-}_{0})=k^{\prime}\overline{\nn}\left(X_{t}\in \ud x, t<\zeta\right)\quad \text{on } (0,\infty),\ \text{for each } t>0.$$ One of our main results estates that this {is} indeed  the case. Later, we will focus on describing not only the entrance law but $\nn$ on $\mathcal{F}_{t}$ for each $t.$ For that end, it will be crucial to understand the sign of the excursion near to its starting point under the measure $\nn$. 

Let $\ee_{q}$ be an independent exponential time of parameter $q\geq 0,$ $g_{t}=\sup\{s\leq t: X_{s}=0\}$ and $G_{q}=g_{\ee_{q}}.$ Define the events
$$\mathcal{A}^{+}=\{X_{G_{q}+t}>0 \ \text{for all sufficiently small } 0<t<\ee_{q}-G_{q}\},$$
$$\mathcal{A}^{-}=\{X_{G_{q}+t}<0 \ \text{for all sufficiently small } 0<t<\ee_{q}-G_{q}\}.$$
{P. W. Millar proved, in  \cite{millar}, that for any SNLP without Gaussian component, i.e. $\sigma^{2}=0,$} $\p(\mathcal{A}^{+})=1,$ and $\p(\mathcal{A}^{-})=0;$ while if $\sigma^{2}>0$ then $\p(\mathcal{A}^{+}), \p(\mathcal{A}^{-})>0$ and $\p(\mathcal{A}^{+})+\p(\mathcal{A}^{-})=1.$ An application of the master formula gives
the equality
\begin{equation*}
\p(\mathcal{A}^{+})=\frac{\nn\Big(\{X_{t}>0 \ \text{for all sufficiently small } 0<t<\ee_{q}\}, \ee_{q}<\zeta\Big)}{\nn\left(\ee_{q}<\zeta\right)},
\end{equation*}
and 
\begin{equation*}
\p(\mathcal{A}^{-})=\frac{\nn\Big(\{X_{t}<0 \ \text{for all sufficiently small } 0<t<\ee_{q}\}, \ee_{q}<\zeta\Big)}{\nn\left(\ee_{q}<\zeta\right)}.
\end{equation*}
We deduce therefrom that if $\sigma^{2}=0$, under $\nn$, {the excursions always start  from  $0$ from above, and if $\sigma^{2}>0$ then some excursions start from $0$ from above  and others from  below, but there are no excursions oscillating from  above and below zero immediately after starting.} In any case, because of the regularity of zero, under $\nn$, the excursions start from $0$ continuously, $\nn(X_{0+}\neq 0)=0,$ see e.g. (2.9) in \cite{millar} or Theorem 2.7 in \cite{yano-exc}. 

In the following Lemma, which we find interesting in itself, we determine the rate of excursions that last longer than an independent exponential r.v. and never  change sign during its life time. We also describe those excursions that change sign before or after an exponential time. 
\begin{lemma}\label{Lemma:intensities}
For any $\beta>0,$ we have the following identities
\begin{itemize}
\item[(i)] $\displaystyle \nn(\zeta>\mathbf{e}_{\beta})=\frac{1}{\Phi^{\prime}(\beta)},$ and $\displaystyle \nn(\zeta=\infty)=\frac{1}{\Phi^{\prime}(0+)}=\Psi^{\prime}(\Phi(0)+);$
\item[(ii)] $\displaystyle \nn(\ee_{\beta}<\zeta=\tau^{-}_0<\infty)=\frac{\sigma^{2}}{2}\left(\Phi(\beta)-\Phi(0)\right),$ and $\nn(\tau^{-}_{0}=\infty)=0,$ unless $X$ drifts towards $+\infty,$ in which case $$\displaystyle \nn(\tau^{-}_{0}=\infty)=\Psi^{\prime}(0+);$$
\item[(iii)] $\displaystyle \nn(0<\tau^{-}_0<\ee_{\beta}<\zeta)=\Phi(\beta)\int^{\infty}_{0}e^{-\Phi(\beta)u}u{\overline{\Pi}^{-}(u)}\ud u ,$ and $\displaystyle \nn(0<\tau^{-}_0<\zeta=\infty)=0,$ unless $X$ drifts towards $-\infty$, in which case $$\nn(0<\tau^{-}_0<\zeta=\infty)=\Phi(0)\int^{\infty}_{0}e^{-\Phi(0)u}u{\overline{\Pi}^{-}(u)}\ud u ;$$
\item[(iv)] $\displaystyle \nn(\tau^{-}_0=0, \ee_{\beta}<\zeta<\infty)=\frac{\sigma^{2}}{2}\left(\Phi(\beta)-\Phi(0)\right),$ and $\displaystyle \nn(\tau^{-}_{0}=0, \zeta=\infty)=0,$ unless $X$ drifts towards $-\infty,$ in which case $$\displaystyle \nn(\tau^{-}_{0}=0, \zeta=\infty)=\frac{\sigma^{2}}{2}\Phi(0);$$ 
\item[(v)] $\displaystyle \nn(\ee_{\beta}<\tau^{-}_0<\zeta)=\int^{\infty}_{0}\overline{\Pi}^{-}(y)\left(e^{-\Phi(0)y}-e^{-\Phi(\beta)y}\right)\ud y.$
\end{itemize}
\end{lemma}
The proof of this result will be given in Section~\ref{exc-calcul}.  We have now all the elements to state our final main theorem.
\begin{theorem}\label{thm3}
The excursion measure $\nn$ is carried  by the union of the following disjoint sets $$\mathcal{E}_{-}:=\{\tau^{-}_0=0<\zeta\},\quad \mathcal{E}_{+}:=\{0<\tau^{-}_0=\zeta\},\quad \mathcal{E_{\pm}}:= \{0<\tau^{-}_0<\zeta\},$$ $\nn(\mathcal{E}_{-})>0$ and $\nn(\mathcal{E}_{+})>0,$ if and only if $\sigma^{2}>0;$ while $\nn(\mathcal{E}_{\pm})>0$ if and only if $\Pi\neq 0.$ We have the following descriptions.  
\begin{itemize} 
\item[(i)] For any $t>0,$ 
\[
\displaystyle \mathbf{1}_{\{\tau^{-}_{0}=0\}}\nn\Big|_{\mathcal{F}_{t}\cap\{t<\zeta\}}=\frac{\sigma^{2}}{2}\widehat{\overline{\nn}}\Big|_{\mathcal{F}_{t}\cap\{t<\zeta\}},
\]
\item[(ii)] For any $t>0,$
\[
\displaystyle \mathbf{1}_{\{\tau^{-}_{0}=\zeta<\infty\}}\nn\Big|_{\mathcal{F}_{t}\cap\{t<\zeta\}}={\frac{\sigma^{2}}{2}}W^{\prime}(X_{t})\underline{\nn}\Big|_{\mathcal{F}_{t}\cap\{t<\zeta\}},
\]
\item[(iii)]  for any $t>0$, $s\geq 0,$ and functionals $F\in\mathcal{H}_{t}$, $G\in \mathcal{H}_{s},$ $f:\mathbb{R}^{2}\to \mathbb{R}^{+}$ measurable
\begin{equation*}
\begin{split}
&\nn\left(F_{t}\cap\{t<\tau^{-}_{0}, X_{\tau^{-}_{0}}<0\} f(X_{\tau^{-}_{0}-}, X_{\tau^{-}_{0}})G\circ\theta_{\tau^{-}_{0}}, \tau^{-}_{0}+s<T_{0}\right)={}\underline{\nn}\left(F K(fH)(X_{t}),  t<\zeta\right),
\end{split}
\end{equation*}
where $H(z)=\e^{0}_{z}\left(G, s<T_{0}\right),$ for $z<0$ and $K(x,\ud y, \ud z)$ is the kernel defined in (\ref{KB}).
\end{itemize}
\end{theorem}
 
\section{Proof of Theorems \ref{teo:1} and \ref{theo:2}}\label{proofofthm12}  
\subsection{Some elementary results for $h_{\beta}$}
Here we establish some elementary facts for the function $h_{\beta}$ that will be useful hereafter. The first Lemma describes the behaviour of $h_{\beta}$ near zero. 
\begin{lemma}\label{lemma:h0}
The function $h_{\beta}$ satisfies that
\begin{equation*}\label{eq:12}
\lim_{x\uparrow0}\frac{h_{\beta}(x)}{|x|}=\Phi^{\prime}(\beta)\Phi(\beta),
\end{equation*}
and 
\begin{equation*}
\lim_{x\downarrow 0}\frac{h_{\beta}(x)}{W(x)}=1-\frac{\sigma^{2}}{2}\Phi^{\prime}(\beta)\Phi(\beta).
\end{equation*}

\end{lemma}
\begin{proof}
The first claim follows from the fact that $W^{(\beta)}\equiv0$ in $(-\infty,0].$  Moreover, according to the series representation for scale functions,
{$$W^{(\beta)}(x)=\sum_{j\geq 0}\beta^{j}W^{*(j+1)}(x),$$
where $W^{*(k)}$ denotes the $k$-th convolution of $W$ with itself, 
 see for instance identity (3.6) in page 132 of \cite{KKR}, and the inequality $$W^{*(k+1)}(x)\leq x W(x)W^{*(k)}(x),\qquad x>0.$$ We deduce that $W^{(\beta)}(x)/W(x)\to1$ as $x\downarrow 0$, which together with the fact { $W^{\prime}(0+)=\frac{2}{\sigma^{2}}$}, see for instance Lemma 3.2 in \cite{KKR}, easily leads to the second assertion. }
\end{proof}

In the following Lemma we  describe the limit of $h_{\beta}$ as $\beta\to0.$ For that end, we will need the constant ${\bf A}$ defined below 
\begin{equation}\label{A}
{\bf A}:=\lim_{\beta\to 0}\Phi^{\prime}(\beta)\Phi(\beta)=\begin{cases}
\frac{\Phi(0)}{\Psi^{\prime}(\Phi(0))}, &\ \text{if } \Psi^{\prime}(0+)<0,\\
\frac{1}{\Psi^{\prime\prime}(0+)}, &\ \text{if } \Psi^{\prime}(0+)=0,\\
0, &\ \text{if } \Psi^{\prime}(0+)>0.
\end{cases}
\end{equation}

\begin{lemma}\label{lemma:20}The following limits hold
\begin{itemize}
\item[(i)] for $x<0,$ 
\begin{equation*}
\frac{h_{\beta}(x)}{\Phi(\beta)\Phi^{\prime}(\beta)}\uparrow_{\beta\to 0} \frac{1-e^{\Phi(0)x}}{\Phi(0)}=g^{-}(x),
\end{equation*}
{where the right-hand side} is interpreted in the limiting sense when $\Phi(0)=0,$ equivalently $\Psi^{\prime}(0+)\geq 0,$ and it is hence equal to $|x|;$
\item[(ii)] for $x>0,$
\begin{equation*}\label{eq21-0}
\lim_{\beta\to 0}\frac{h_{\beta}(x)}{\Phi(\beta)\Phi^{\prime}(\beta)}=g(x)+\frac{1}{{\bf A}}W(x), 
\end{equation*}
where we understand $\frac{1}{\bf A}=\infty$ if ${\bf A}=0.$ 
\item[(iii)] for $x\in \mathbb{R}$, 
\begin{equation*}\label{eq22-0}
\widetilde{g}(x):=\lim_{\beta\to 0}\frac{h_{\beta}(x)\Phi(\beta)}{\beta\Phi^{\prime}(\beta)}= \begin{cases} g^{+}(x)\mathbf{1}_{\{x>0\}}-\frac{1}{\Psi^{\prime\prime}(0+)}x, & \text{if}\ \Psi^{\prime}(0+)=0,\\
g^{+}(x)\mathbf{1}_{\{x>0\}}, & \text{if}\ \Psi^{\prime}(0+)>0,\\
\infty,& \text{if}\ \Psi^{\prime}(0+)<0. 
\end{cases} 
\end{equation*}

\end{itemize}
\end{lemma}
\begin{proof} The claim in (i) {follows from  the definition of $h_{\beta}$ and the fact that the function} $\lambda \mapsto (1-e^{-\lambda})/\lambda$ is non-increasing. The proof of (ii) and (iii) in this lemma is elementary, we just need to remark the {following} estimate as $\beta \to 0,$
$$\frac{\beta}{\Phi^{2}(\beta)}\sim\begin{cases}
\frac{\Psi^{\prime}(\Phi(\beta))}{\Phi(\beta)}, & \text{if} \ \Psi^{\prime}(0+)\geq 0,\\
\frac{\beta}{\Phi^{2}(0)}, & \text{if} \ \Psi^{\prime}(0+)< 0, 
\end{cases}$$
{and observe that the following identity holds}
\begin{equation*}\label{eq:chin1}
h_{\beta}(x)=\frac{\Phi(\beta)}{\beta\Phi^{\prime}(\beta)}g^{+}_{\beta}(x)\mathbf{1}_{\{x>0\}}+\frac{\Phi^2(\beta)}{\beta}\frac{(1-e^{y\Phi(\beta)})}{\Phi(\beta)}.
\end{equation*} \end{proof}

\subsection{Proof of Theorem \ref{teo:1}}
\begin{proof}[Proof of (i) in Theorem \ref{teo:1}]
{Since  $X$ is spectrally negative, we observe that its transition probabilities when issued from $x<0$ satisfy} $$P^{\beta,0}_{t}(x,\ud y)=\e^{\beta}_{x}(X_{t}\in \ud y, t<T_{\{0\}})=\e_{x}(X_{t}\in \ud y, t<\tau^{+}_{0}\wedge\mathbf{e}_{\beta})=e^{-\beta t}\widehat{\mathbb{E}}_{-x}\left(X_{t}\in -\ud y, t<\tau^{-}_0\right).$$ From the results in \cite{CD}, we obtain the convergence $$\lim_{x\uparrow 0}\frac{1}{|x|}\e_{x}(F, t<\tau^{+}_{0}\wedge\mathbf{e}_{\beta})=\widehat{\overline{\nn}}(F, t<\zeta\wedge\mathbf{e}_{\beta}),$$ for  $F\in\mathcal{C}_{t},$ $ t\ge 0$. {On the other hand, since $X$ has paths of  unbounded variation, the function $h_{\beta}$ is  continuous and bounded.} Indeed, $W^{(\beta)}$ is continuous at $0,$ by Lemma 3.1 in \cite{KKR}, and hence it is continuous everywhere owing to the fact that it is also a subadditive function. The result now follows from an application of  Lemma~\ref{lemma:h0}, as the following calculation now shows
\begin{equation*}
\begin{split}
\e^{\updownarrow,\beta}_{x}\left(F, t<\zeta\right)&=\frac{|x|}{h_{\beta}(x)}\frac{\e^{\beta,0}_{x}\left(F, t<\zeta, T_{0}=\infty\right)}{|x|}\\
&=\frac{|x|}{h_{\beta}(x)}\frac{\e^{\beta,0}_{x}\left(Fh_{\beta}(X_{t}), t<\zeta\right)}{|x|}\\
&=\frac{|x|}{h_{\beta}(x)}\frac{\e_{x}\left(Fh_{\beta}(X_{t}), t<T_{[0,\infty)}\wedge \mathbf{e}_{\beta}\right)}{|x|}\\
&\xrightarrow[x\uparrow 0]{}\frac{1}{\Phi^{\prime}(\beta)\Phi(\beta)}\widehat{\overline{\nn}}(Fh_{\beta}(X_{t}), t<\zeta\wedge\mathbf{e}_{\beta}).
\end{split}
\end{equation*}
\end{proof}
For the proof of (ii) in Theorem \ref{teo:1}, we  need the following lemma.
\begin{lemma}\label{Lemma3} For any $t>0$ and $F\in \mathcal{C}_{t}$, the following convergence holds
\begin{equation*}\label{conv001}
\lim_{x\downarrow 0}\frac{1}{W(x)}\e^{\beta,0}_{x}\left(F, t<\tau_{0}^{-}\leq \zeta\right)=\underline{\nn}\left(F, t<\zeta<\mathbf{e}_{\beta}\right).
\end{equation*}
\end{lemma}
\begin{proof}
Observe that for $x>0, t>0$ and $F\in\mathcal{C}_{t}$ we have the equality 
\begin{equation*}
\begin{split}
\p_x^{\beta,0}(F, t<\tau^{-}_0\leq \zeta)&=\p_x(F, t<\tau^{-}_0\leq T_{0}\wedge \mathbf{e}_{\beta})\\
&=\e_x(F e^{-\beta \tau^{-}_0}, t<\tau^{-}_0)\\
&=\e_x\left(F\e_{X_{t}}\left(e^{-\beta \tau^{-}_0}\right)e^{-\beta t}, t<\zeta\right),
\end{split}
\end{equation*}
{where in the second equality we used the fact that when the process starts from a positive state, then  $\tau^{-}_0\leq T_{0}$, a.s.; and in the third equality, we used the Markov property.} Recall the identity
$$\e_{y}\left(e^{-\beta \tau^{-}_0}\mathbf{1}_{\{\tau^-_0<\infty\}}\right)=Z^{(\beta)}(y)-\frac{\beta}{\Phi(\beta)}W^{(\beta)}(y),\quad\textrm{where} \quad Z^{(\beta)}(y)=1+\beta\int^{y}_{0}W^{(\beta)}(z)\ud z,\,\,  y\in \mathbb{R},$$ which can be found in Theorem 2.6 in \cite{KKR}. {Since the process has paths of unbounded variation, the  function  from above }is continuous and bounded. Using these observations together with Proposition 2 in \cite{rd}, it is easy to check  that the following limit holds
\begin{equation*}\label{eq:27}
\begin{split}
\lim_{x\downarrow 0}\frac{1}{W(x)}\e_x^{\beta,0}\left(F, t<\tau^{-}_0\leq \zeta\right)=\underline{\nn}\left(F\e_{X_{t}}\left(e^{-\beta \tau^{-}_0}\right)e^{-\beta t}, t<\zeta\right),\\
\end{split}
\end{equation*}
for  $F\in\mathcal{C}_{t}$,  $t>0.$ Furthermore, using that under $\underline{\nn}$ the canonical process is Markovian  with the same semigroup as the original L\'evy process killed at the first passage time below zero, we get the identity
\begin{equation*}
\begin{split}
\underline{\nn}\left(F\e_{X_{t}}\left(e^{-\beta \tau^{-}_0}\right)e^{-\beta t}, t<\zeta\right)=\underline{\nn}(F e^{-\beta\zeta}, t<\zeta)=\underline{\nn}(F, t<\zeta<\ee_{\beta}).\\
\end{split}
\end{equation*}
{This implies that  the convergence in the statement holds.}\end{proof}

\begin{proof}[Proof of (ii) in Theorem \ref{teo:1}]
Let $F$ and $G$ be as in the statement of Theorem~\ref{teo:1}.  Observe that the Markov property applied at the stopping time $\tau^{-}_{0}$ implies that the expression on the left hand side of (\ref{th1:02}) can be written as
\begin{equation}\label{eq:2.19}
\begin{split}
\e^{\updownarrow,\beta}_{x}&\Big(F f(X_{\tau^{-}_0-}, X_{\tau^{-}_0} )G\circ\theta_{\tau^{-}_{0}}, t<\tau^{-}_0<\tau^{-}_0+s< \zeta\Big)\\
&=\frac{\e^{\beta,0}_{x}\Big(F f(X_{\tau^{-}_0-}, X_{\tau^{-}_0} )G\circ\theta_{\tau^{-}_{0}}, t<\tau^{-}_0<\tau^{-}_0+s< \zeta, T_{0}=\infty\Big)}{\p^{\beta,0}_{x}(T_{0}=\infty)}\\
&=\frac{\e^{\beta,0}_{x}\left(F f(X_{\tau^{-}_0-}, X_{\tau^{-}_0} )\mathbf{1}_{\{X_{\tau^{-}_{0}}<0\}}\e^{\beta,0}_{X_{\tau^{-}_{0}}}\left(G,s<\zeta, T_{0}=\infty\right), t<\tau^{-}_0< \zeta\right)}{\p^{\beta,0}_{x}(T_{0}=\infty)}\\
&=\frac{1}{u_{\beta}(0)}\frac{\e^{\beta,0}_{x}\left(F f(X_{\tau^{-}_0-}, X_{\tau^{-}_0} )h_{\beta}(X_{\tau^{-}_{0}})\mathbf{1}_{\{X_{\tau^{-}_{0}}<0\}}\e^{\beta,0}_{X_{\tau^{-}_{0}}}\left(G,s<\zeta | T_{0}=\infty\right), t<\tau^{-}_0< \zeta\right)}{\p^{\beta,0}_{x}(T_{0}=\infty)}\\
&=\frac{\e^{\beta,0}_{x}\left(F f(X_{\tau^{-}_0-}, X_{\tau^{-}_0} )h_{\beta}(X_{\tau^{-}_{0}})\mathbf{1}_{\{X_{\tau^{-}_{0}}<0\}}\e^{\beta,0}_{X_{\tau^{-}_{0}}}\left(G,s<\zeta | T_{0}=\infty\right), t<\tau^{-}_0< \zeta\right)}{h_{\beta}(x)}.
\end{split}
\end{equation}
Moreover, the assumptions on $G$ imply that the function 
$$x\mapsto\e^{\beta,0}_{x}\left(G,s<\zeta | T_{0}=\infty\right),$$ is a continuous and bounded function. A consequence of the former identity and this observation is that we can restrict ourselves to the case where $s=0$ and $G=1.$ So, it is enough to analyse the expression 
\begin{equation*}
\begin{split}
&\frac{1}{h_{\beta}(x)}\e^{\beta,0}_{x}\left(F f(X_{\tau^{-}_0-}, X_{\tau^{-}_0} )h_{\beta}(X_{\tau^{-}_{0}})\mathbf{1}_{\{X_{\tau^{-}_{0}}<0\}}, t<\tau^{-}_0< \zeta\right).
\end{split}
\end{equation*}
And it is worth observing that the latter expression coincides with
$$\e^{\beta,0}_{x}\left(F f(X_{\tau^{-}_0-}, X_{\tau^{-}_0} )\mathbf{1}_{\{X_{\tau^{-}_{0}}<0\}}, t<\tau^{-}_0< \zeta \Big | T_{0}=\infty\right),$$ as it can be easily verified taking $s=0$ {and $G=1$} in (\ref{eq:2.19}).
As a further consequence of the Markov property and  identity (\ref{KB}), we have
\begin{equation*}
\begin{split}
&\e^{\beta,0}_{x}\left(Fh_{\beta}(X_{\tau^{-}_{0}})f(X_{\tau^{-}_0-}, X_{\tau^{-}_0})\mathbf{1}_{\{X_{\tau^{-}_{0}-}>0\}}, t<\tau^{-}_0<\zeta\right)=\e_{x}\Big(FK_{\beta}(fh_{\beta})(X_{t})e^{-\beta t}, t<\tau^{-}_0\Big).
\end{split}
\end{equation*}
The assumption of unbounded variation {paths} imply that for any $f:\mathbb{R}^{2}\to \mathbb{R}$ continuous and bounded, the mapping  $x\mapsto K_{\beta}(fh_{\beta})(x),$ is continuous and bounded in $(0,\infty).$  The above facts allow us to infer the following series of equalities 
\begin{equation}\label{teo1:3}
\begin{split}
\lim_{x\downarrow 0}\e^{\updownarrow,\beta}_{x}&(F f(X_{\tau^{-}_0-}, X_{\tau^{-}_0} ), t<\tau^{-}_0\leq \zeta)\\
&=\lim_{x\downarrow 0}\e^{\beta,0}_{x}(Ff(X_{\tau^{-}_0-}, X_{\tau^{-}_0} ), t<\tau^{-}_0\leq \zeta | T_{0}=\infty)\\
&=\lim_{x\downarrow 0}\frac{1}{h_{\beta}(x)}\e^{\beta,0}_{x}\left(Ff(X_{\tau^{-}_0-}, X_{\tau^{-}_0} )\mathbf{1}_{\{X_{\tau^{-}_{0}}<0\}}h_{\beta}(X_{\tau^{-}_{0}}), t<\tau^{-}_0< \zeta\right)\\
&=\frac{1}{\left(1-\frac{\sigma^{2}}{2}\Phi^{\prime}(\beta)\Phi(\beta)\right)}\lim_{x\downarrow 0}\frac{1}{W(x)}\e^{\beta,0}_{x}\left(Fh_{\beta}(X_{\tau^{-}_{0}})f(X_{\tau^{-}_{0}-}, X_{\tau^{-}_0}), t<\tau^{-}_0<\zeta\right).
\end{split}
\end{equation}
{Observe that the limit in the right hand side of the previous identity is equal to
\begin{equation*}
\begin{split}
&\lim_{x\downarrow 0}\frac{1}{W(x)}\e^{\beta,0}_{x}\left(F K_{\beta}(h_{\beta}f)(X_{t}), t<\tau^{-}_0<\zeta\right)=\underline{\nn}(FK_{\beta}(h_{\beta}f)(X_{t})e^{-\beta t}, t<\zeta).
\end{split}
\end{equation*}}
We conclude the proof by applying the Markov property under the measure $\underline{\nn},$ to show that  {the right hand side of the above is equal to} 
\[
\underline{\nn}(Fh_{\beta}(X_{\zeta})f(X_{\zeta-}, X_{\zeta} )\mathbf{1}_{\{X_{\zeta-}>0\}}, t<\zeta< \ee_{\beta}).
\]
\end{proof}
\begin{proof}[Proof of (iii) in Theorem \ref{teo:1}]
We start by observing {that for $x\in \mathbb{R}\setminus\{0\}$, the following identity holds 
\begin{equation*}
\begin{split}
\p^{\updownarrow,\beta}_{x}\left(\tau^{-}_{0}=\infty\right)=\frac{\e^{\beta,0}_{x}\left(\tau^{-}_0=\infty\right)}{\p^{\beta,0}_{x}(T_{0}=\infty)}=\frac{\e_{x}\left(1-e^{-\beta \tau^{-}_{0}}\right)}{\p^{\beta,0}_{x}(T_{0}=\infty)}. 
\end{split}
\end{equation*}
The latter is  consequence of the definition of $\p^{\updownarrow,\beta}$ as the measure $\p^{\beta,0}$ conditioned on the event $\{T_{0}=\infty\}.$ Using the previous identity, the Markov property  and  Proposition 2 in \cite{rd}, we deduce}
\begin{equation*}
\begin{split}
\e^{\updownarrow,\beta}_{x}\left(F, t<\zeta, \tau^{-}_0=\infty\right)&=\e^{\updownarrow,\beta}_{x}\left(F\mathbf{1}_{\{X_{t}>0, t<\zeta\}} \p^{\updownarrow,\beta}_{X_{t}}\left(\tau^{-}_0=\infty\right)\right)\\
&=\frac{1}{\p^{\beta,0}_{x}(T_{0}=\infty)}\e_{x}\left(F\mathbf{1}_{\{X_{t}>0, t<T_{0}\wedge\ee_{\beta}\}} \p^{\updownarrow,\beta}_{X_{t}}\left(\tau^{-}_0=\infty\right)\p^{\beta,0}_{X_{t}}(T_{0}=\infty)\right)\\
&=\frac{1}{\p^{\beta,0}_{x}(T_{0}=\infty)}\e_{x}\left(F\mathbf{1}_{\{X_{t}>0, t<T_{0}\}}e^{-\beta t}\e_{X_{t}}\left(1-e^{-\beta \tau^{-}_0}\right)\right)\\
&=\frac{1}{\p^{\beta,0}_{x}(T_{0}=\infty)}\e_{x}\left(F\mathbf{1}_{\{t<\tau^{-}_0\}}e^{-\beta t}\e_{X_{t}}\left(1-e^{-\beta \tau^{-}_0}\right)\right)\\
&\xrightarrow[x\downarrow 0]{}\frac{\Phi^{\prime}(\beta)}{\left(1-\frac{\sigma^{2}}{2}\Phi^{\prime}(\beta)\Phi(\beta)\right)}\underline{\nn}\left(F\e_{X_{t}}\left(1-e^{-\beta \tau^{-}_0}\right)e^{-\beta t}, t<\zeta\right)\\
&=\frac{\Phi^{\prime}(\beta)}{\left(1-\frac{\sigma^{2}}{2}\Phi^{\prime}(\beta)\Phi(\beta)\right)}\underline{\nn}\left(F, t<\ee_{\beta}<\zeta\right),
\end{split}
\end{equation*}
for every $t>0,$ and $F\in\mathcal{C}_{t}.$
\end{proof}
\begin{lemma}\label{lem:1}
Let $X$ be a SNLP {with unbounded variation paths}. We have the following convergence
\begin{itemize}
\item[(i)] for any $x<0$, $t>0$ and for any $F\in \mathcal{C}_{t}$
\begin{equation*}
\e^{\updownarrow}_{x}(F,t<\zeta):=\lim_{\beta\downarrow0}\e^{\updownarrow,\beta}_{x}(F,t<\zeta)=\e^{\downarrow}_x\left(F, \ t<\zeta\right)\end{equation*}
\item[(ii)] for any $x, t>0,$ $s\geq 0$, $F\in\mathcal{C}_{t},$ $G\in\mathcal{C}_{s}$ and $f:\mathbb{R}^{2}\to\mathbb{R}$ continuous and bounded
\begin{equation*}\label{lema3:02}
\begin{split}
\lim_{\beta\downarrow 0}&\e^{\updownarrow,\beta}_{x}\left(F f(X_{\tau^{-}_0-}, X_{\tau^{-}_0} )G\circ\theta_{\tau^{-}_{0}}, t<\tau^{-}_0<\tau^{-}_0+s< \zeta\right)\\
&=\frac{{\bf A}}{{\bf A} g(x)+W(x)}
\e_{x}\left(F f(X_{\tau^{-}_0-}, X_{\tau^{-}_0} )g^{-}(X_{\tau^{-}_{0}})\mathbf{1}_{\{X_{\tau^{-}_{0}}<0\}}\e^{\updownarrow}_{X_{\tau^{-}_{0}}}\left(G,s<\zeta\right), t<\tau^{-}_0< \zeta\right),
\end{split}
\end{equation*}
 where the constant ${\bf A}$ is defined in (\ref{A}). The {right-hand side} in the above equation is equal to zero if $\Psi^{\prime}(0+)>0$ or the variance of $X$ is infinite. We also have 
\begin{equation}\label{condlaw}
\begin{split}
\lim_{\beta\to0}\e^{\updownarrow,\beta}_{x}&\left(F f(X_{\tau^{-}_0-}, X_{\tau^{-}_0} )G\circ\theta_{\tau^{-}_{0}}, t<\tau^{-}_0\leq \tau^{-}_0+s<\zeta\Big |\ \tau^{-}_{0}<\zeta\right)\notag\\
&=\frac{\e_{x}\left(F f(X_{\tau^{-}_0-}, X_{\tau^{-}_0} )g^{-}(X_{\tau^{-}_{0}})\mathbf{1}_{\{X_{\tau^{-}_{0}}<0\}}{\e^{\updownarrow}_{X_{\tau^{-}_{0}}}}\left(G,s<\zeta\right), t<\tau^{-}_0< T_{0})\right)}{\e_{x}\left(g^{-}\left(X_{\tau^{-}_{0}}\right), \tau^{-}_{0}<T_{0}\right)}.\notag
\end{split}
\end{equation}
\item[(iii)] For any $x, t>0,$ and  $F\in \mathcal{C}_{t}$ 
\begin{equation*}
\begin{split}
&\lim_{\beta\downarrow0}\e^{\updownarrow,\beta}_{x}(F, t<\zeta | \tau^{-}_0=\infty)=\e^{\uparrow}_{x}\left(F\mathbf{1}_{\{t<\zeta\}}\right).
\end{split}
\end{equation*}
\end{itemize}
\end{lemma}

\begin{proof}[Proof of  Lemma \ref{lem:1}] {We first start with the  proof of part (i). }Let $F\in\mathcal{C}_{t}.$ For any $x<0$, we have the following relation
\begin{align}
\e_x^{\updownarrow,\beta}(F;t<\zeta)=\frac{1}{h_{\beta}(x)}\e_x^{\beta,0}(Fh_{\beta}(X_t))=\frac{1}{h_{\beta}(x)}\e_x\Big(F\mathbf{1}_{\{t<\zeta\}}h_{\beta}(X_t);t<T_0\wedge\mathbf{e}_{\beta}\Big).\notag
\end{align} Using that $\Phi$ is increasing, it is easily veryfied that the function
$$
J(\beta,x):=\frac{h_{\beta}(x)}{\Phi(\beta)\Phi'(\beta)}=\frac{(1-e^{\Phi(\beta)x})}{\Phi(\beta)},
$$
is decreasing as a function of $\beta>0$, for each fixed $x\leq0$. Therefore by monotone convergence we obtain 
\begin{align*}
\lim_{\beta\downarrow0}\e^{\updownarrow,\beta}_{x}(F,t<\zeta)&=\lim_{\beta\downarrow0}\frac{1}{h_{\beta}(x)}\e_x(Fh_{\beta}(X_t);t<T_0\wedge\mathbf{e}_{\beta})\notag\\
&=\lim_{\beta\downarrow0}\frac{\Phi(\beta)}{1-e^{\Phi(\beta)x}}\e_x\left(F\frac{1-e^{\Phi(\beta)X_t}}{\Phi(\beta)};t<T_0\right)\notag\\
&=\e_x\left(F\frac{g^{-}(X_{t})}{g^{-}(x)};t<\zeta\right)\\
&=\e_{x}^{\downarrow}\left(F, t<\zeta\right),\notag
\end{align*}
{where in the last equality, we used the identity  (\ref{eq:cond-}).}

{We now proceed with the proof of part (ii).} Let $F,$ $G$ and $f$ be as in the statement of Lemma \ref{lem:1}. The identity (\ref{eq:2.19}) in the proof of (ii) \textrm{in} Theorem 1 implies that for $x>0$
\begin{equation}
\begin{split}
\e^{\updownarrow,\beta}_{x}&\left(F f(X_{\tau^{-}_0-}, X_{\tau^{-}_0} )G\circ\theta_{\tau^{-}_{0}}, t<\tau^{-}_0\leq \tau^{-}_0+s< \zeta\right)\notag\\
&=\frac{1}{h_{\beta}(x)}\e^{\beta,0}_{x}\left(F f(X_{\tau^{-}_0-}, X_{\tau^{-}_0} )h_{\beta}(X_{\tau^{-}_{0}})\mathbf{1}_{\{X_{\tau^{-}_{0}}<0\}}\e^{\updownarrow,\beta}_{X_{\tau^{-}_{0}}}\left(G,s<\zeta\right), t<\tau^{-}_0< \zeta\right).\notag
\end{split}
\end{equation}
Also from the part (i) in Lemma \ref{lem:1}, we have 
$$
\lim_{\beta\downarrow0} \e^{\updownarrow,\beta}_{X_{\tau^{-}_{0}}}\left(G,s<\zeta\right)\mathbf{1}_{\{X_{\tau^{-}_{0}}<0\}}= {\e^{\updownarrow}_{X_{\tau^{-}_{0}}}}\left(G,s<\zeta\right)\mathbf{1}_{\{X_{\tau^{-}_{0}}<0\}}.
$$
By a dominated convergence argument and (i)-(ii) in Lemma~\ref{lemma:20}, we get 
\begin{align*}
\lim_{\beta\downarrow0}&\e^{\updownarrow,\beta}_{x}\left(F f(X_{\tau^{-}_0-}, X_{\tau^{-}_0} )G\circ\theta_{\tau^{-}_{0}}, t<\tau^{-}_0<\tau^{-}_0+s< \zeta\right)\notag\\
&=\lim_{\beta\downarrow0}\frac{\Phi(\beta)\Phi^{\prime}(\beta)}{h_{\beta}(x)}\e^{\beta,0}_{x}\left(F f(X_{\tau^{-}_0-}, X_{\tau^{-}_0} )\frac{h_{\beta}(X_{\tau^{-}_{0}})}{\Phi(\beta)\Phi^{\prime}(\beta)}\mathbf{1}_{\{X_{\tau^{-}_{0}}<0\}}\e^{\updownarrow,\beta}_{X_{\tau^{-}_{0}}}\left(G,s<\zeta\right), t<\tau^{-}_0< \zeta\right)\notag\\
&=\frac{{\bf A}}{{\bf A} g(x)+W(x)}
\e_{x}\left(F f(X_{\tau^{-}_0-}, X_{\tau^{-}_0} )g^{-}(X_{\tau^{-}_{0}})1_{\{X_{\tau^{-}_{0}}<0\}}{ \e^{\updownarrow}_{X_{\tau^{-}_{0}}}}\left(G,s<\zeta\right), t<\tau^{-}_0< T_{0}\right),
\end{align*} with ${\bf A}$  {defined as in (\ref{A}). Recall that  ${\bf A}=0$ whenever $\Psi^{\prime}(0+)>0,$ or the variance of $X$ is infinite} and hence the latter factor equals zero for all $x>0$. Analogously, performing similar calculations 
we obtain the conditional limit in (ii) in Lemma (\ref{lem:1}), i.e.
\begin{equation*}
\begin{split}
\e^{\updownarrow,\beta}_{x}&\left(F f(X_{\tau^{-}_0-}, X_{\tau^{-}_0} )G\circ\theta_{\tau^{-}_{0}}, t<\tau^{-}_0<\tau^{-}_{0}+s<\zeta | \tau^{-}_0< \zeta\right)\\
&=\frac{\e^{\beta,0}_{x}\left(F f(X_{\tau^{-}_0-}, X_{\tau^{-}_0} )h_{\beta}(X_{\tau^{-}_{0}})\mathbf{1}_{\{X_{\tau^{-}_{0}}<0\}}\e^{\updownarrow,\beta}_{X_{\tau^{-}_{0}}}\left(G,s<\zeta\right), t<\tau^{-}_0< \zeta\right)}{\e^{\beta,0}_{x}\left(h_{\beta}(X_{\tau^{-}_{0}})\mathbf{1}_{\{X_{\tau^{-}_{0}}<0\}},\tau^{-}_0< \zeta\right)}\\
&\xrightarrow[\beta\to0]{}\frac{\e_{x}\left(F f(X_{\tau^{-}_0-}, X_{\tau^{-}_0} )g^{-}(X_{\tau^{-}_{0}})\mathbf{1}_{\{X_{\tau^{-}_{0}}<0\}}{ \e^{\updownarrow}_{X_{\tau^{-}_{0}}}}\left(G,s<\zeta\right), t<\tau^{-}_0< T_{0}\right)}{\e_{x}\left(g^{-}(X_{\tau^{-}_{0}})\mathbf{1}_{\{X_{\tau^{-}_{0}}<0\}}, \tau^{-}_0< T_{0}\right)},
\end{split}
\end{equation*}
{which completes the proof of part (ii).}

{Finally, we prove part (iii).} Let $t>0,$ and $F\in\mathcal{C}_{t}.$ From the results by Chaumont and Doney \cite{CD}, we have $$\lim_{\beta\to 0}\e_x\left(F\mathbf{1}_{\{t<\tau_0^-\}}e^{-\beta t}\frac{\p_{X_t}\left(\tau^{-}_{0}>\mathbf{e}_{\beta}\right)}{\underline{\nn}(\zeta>\mathbf{e}_{\beta})}\right)=\e_x\left(F\mathbf{1}_{\{t<\tau_0^-\}}g^{+}(X_{t})\right).$$ {Since we assume that  $X$ does not drift towards $-\infty$,  we deduce}
$$\frac{\underline{\nn}\left(\zeta>\mathbf{e}_{\beta}\right)}{\p^{\beta,0}_{x}(T_{0}=\infty)}=\frac{u_{\beta}(0)\underline{\nn}\left(\zeta>\mathbf{e}_{\beta}\right)}{h_{\beta}(x)}=\frac{\beta\Phi^{\prime}(\beta)}{\Phi(\beta)h_{\beta}(x)}.$$ The limit of this expression has been calculated {in part (iii) in Lemma (\ref{lemma:20}). Therefore,}  we have the following identity for $x>0$
\begin{equation*}
\begin{split}
\lim_{\beta\to 0}\e^{\updownarrow,\beta}_x(F,t<\zeta,\tau_0^-=\infty)&=\lim_{\beta\to 0}\frac{\underline{\nn}(\zeta>\mathbf{e}_{\beta})}{\mathbb{P}_x^{\beta,0}(T_0=\infty)}\lim_{\beta\to 0}\e_x\left(F\mathbf{1}_{\{t<\tau_0^-\}}e^{-\beta t}\frac{\p_{X_t}\left(\tau^{-}_{0}>\mathbf{e}_{\beta}\right)}{\underline{\nn}(\zeta>\mathbf{e}_{\beta})}\right)\\
&= \frac{g^{+}(x)}{\widetilde{g}(x)}
\e^{\uparrow}_{x}\left(F\mathbf{1}_{\{t<\zeta\}}\right),
\end{split}
\end{equation*}
Moreover, {when $X$ drifts towards $-\infty$, the function $\widetilde{g}$ takes  the value $\infty$ and therefore the latter limit is equals zero. To finish, notice that taking $F, f\equiv 1$, and performing similar calculations to those used in the proof of part (ii),  the limit of the conditional law in part (iii) in Lemma~\ref{lem:1} follows. This completes the proof.}\end{proof}
We {now} proceed  to prove Theorem~\ref{theo:2}. 
\begin{proof}[Proof of Theorem \ref{theo:2}] {We first prove part (i).} Let $t>0,$ and $F\in\mathcal{C}_{t}.$
From {the identity that appears in part (i) in Theorem \ref{teo:1}, we know }
\begin{equation}
\e^{\updownarrow,\beta}_{0-}(F, t<\zeta)= \frac{1}{\Phi(\beta)\Phi^{\prime}(\beta)}\widehat{\overline{\nn}}(F h_{\beta}(X_{t}), t<\zeta\wedge\ee_{\beta}).\notag
\end{equation}
{Hence, the result follows from the observation that for $y<0,$ we have 
$$\frac{h_{\beta}(y)}{\Phi'(\beta)\Phi(\beta)}=\frac{(1-e^{\Phi(\beta) y})}{\Phi(\beta)},$$
 which turns out to be a non-decreasing function of $\beta$.} Therefore by an application of the monotone convergence theorem, we get
\begin{align*}
\lim_{\beta\downarrow0}\e^{\updownarrow,\beta}_{0-}(F, t<\zeta)&=\widehat{\overline{\nn}}(F g^{-}(X_{t}), t<\zeta).
\end{align*}

{For the proof of part (ii), we let $t>0,$ $s\geq 0,$ $F\in\mathcal{C}_{t},$ $G\in\mathcal{G}_{s}$ and $f:\mathbb{R}^{2}\to\mathbb{R},$ be a continuous and bounded function. From Theorem \ref{teo:1}, part (ii),} we have 
\begin{equation*}
\begin{split}
\e^{\updownarrow,\beta}_{0+}&\left(F f(X_{\tau^{-}_0-}, X_{\tau^{-}_0} )G\circ\theta_{\tau^{-}_{0}}, t<\tau^{-}_0<\tau^{-}_0+s< \zeta\right)\\
&\hspace{1cm}=
\frac{1}{\left(1-\frac{\sigma^{2}}{2}\Phi^{\prime}(\beta)\Phi(\beta)\right)}\underline{\nn}\left(Fh_{\beta}(X_{\zeta})f(X_{\zeta-}, X_{\zeta} )\mathbf{1}_{\{X_{\zeta-}>0\}}\e^{\updownarrow,\beta}_{X_{\zeta}}\left(G, s<\zeta\right), {t<\zeta<\ee_{\beta}}\right).
\end{split}
\end{equation*}
{In particular, this implies} 
$$\p^{\updownarrow,\beta}_{0+}\left(\tau^{-}_{0}<\zeta\right)=\frac{\underline{\nn}\left(h_{\beta}(X_{\zeta})\mathbf{1}_{\{X_{\zeta-}>0\}}, {\zeta<\ee_{\beta}}\right)}{\left(1-\frac{\sigma^{2}}{2}\Phi^{\prime}(\beta)\Phi(\beta)\right)}.$$
Using Lemma \ref{lemma:20},  Lemma \ref{lem:1}~(i), and arguing as above, we obtain 
\begin{equation}\label{infvar1}
\begin{split}
\lim_{\beta\downarrow 0}&\frac{\underline{\nn}\left(F\frac{h_{\beta}(X_{\zeta})}{{\Phi(\beta)\Phi'(\beta)}}f(X_{\zeta-}, X_{\zeta} )\mathbf{1}_{\{X_{\zeta-}>0\}}\e^{\updownarrow,\beta}_{X_{\zeta}}\left(G, s<\zeta\right), t<\zeta<\ee_{\beta}\right)}{\underline{\nn}(\frac{h_{\beta}(X_{\zeta})}{{\Phi(\beta)\Phi'(\beta)}}\mathbf{1}_{\{X_{\zeta-}>0\}}, {\zeta<\ee_{\beta}})}\\
&\hspace{3cm}=\frac{\underline{\nn}\left(Fg^{-}(X_{\zeta})f(X_{\zeta-}, X_{\zeta} )\mathbf{1}_{\{X_{\zeta-}>0\}}{\e^{\updownarrow}_{X_{\zeta}}}\left(G, s<\zeta\right), t<\zeta<\infty\right)}{{\underline{\nn}(g^{-}(X_{\zeta})\mathbf{1}_{\{X_{\zeta-}>0\}}, {\zeta<\infty})}}.
\end{split}
\end{equation}
 To conclude, we determine the value of the expression in the denominator of the above equation. As under $\nn,$ the Poissonian structure of the jumps of $X$ implies 
\begin{equation*}
\begin{split}
\underline{\nn}\left(g^{-}(X_{\zeta})\mathbf{1}_{\{X_{\zeta-}>0\}}, {\zeta<\infty}\right)&=\underline{\nn}\left(\sum_{0<t<\zeta}g^{-}(X_{t-}+(X_{t}-X_{t-}))\mathbf{1}_{\{X_{t}-X_{t-}\neq 0, X_{t}<0\}}\right)\\
&=\underline{\nn}\left(\int^{\zeta}_{0}\ud t\int^{0}_{-\infty}\Pi(\ud y)g^{-}(X_{t-}+y)\mathbf{1}_{\{X_{t-}+y< 0\}}\right).
\end{split}
\end{equation*}
Exercise 5 in page 183 in \cite{B}  {tell us that the following identity holds}  
\[\underline{\nn}\left(\int^{\zeta}_{0}\mathbf{1}_{\{X_{t-}\in dy\}}\ud t\right)=e^{-\Phi(0)y}\mathbf{1}_{\{y>0\}}\ud y.
\]  Using this, the definition of $g^{-}$ and making some elementary calculations, we {deduce}
\begin{equation*}
\begin{split}
\underline{\nn}(g^{-}(X_{\zeta})\mathbf{1}_{\{X_{\zeta-}>0\}}, {\zeta<\infty})&=\int^{\infty}_{0}\ud y e^{-\Phi(0)y}\int^{\infty}_{0}\Pi(-\ud z)\frac{1-e^{\Phi(0)(y-z)}}{\Phi(0)}\mathbf{1}_{\{y<z\}}\\
&=\int^{\infty}_{0}e^{-\Phi(0)u}u\overline{\Pi}^{-}(u)\ud u.
\end{split}
\end{equation*}
This quantity is infinite if $X$ does not drift to $-\infty$ and has infinite variance. In this case the quotient in (\ref{infvar1}) is hence equal to zero. {Observe that such quantity is well defined since} for $t>0,$ $\underline{\nn}(t<\zeta)<\infty.$

{Now, we prove part (iii).} Let $t>0$ and assume that $F\in\mathcal{C}_{t}$ and $|F|\leq 1.$
Recall {that} in Theorem \ref{teo:1}, we established 
\begin{equation*}
\begin{split}
&\e^{\updownarrow,\beta}_{0+}(F, t<\zeta, \tau^{-}_0=\infty)=\frac{u_{\beta}(0)}{1-\frac{\sigma^{2}}{2}\Phi(\beta)\Phi^{\prime}(\beta)}{\underline{\nn}}(F, t<\ee_{\beta}<\zeta),
\end{split}
\end{equation*}
{which implies
$$\p^{\updownarrow,\beta}_{0+}(\tau^{-}_0=\infty)=\frac{u_{\beta}(0)}{1-\frac{\sigma^{2}}{2}\Phi(\beta)\Phi^{\prime}(\beta)}{\underline{\nn}}(\ee_{\beta}<\zeta).$$ 
Therefore applying the Markov property, we deduce }
\begin{equation*}
\begin{split}
&\e^{\updownarrow,\beta}_{0+}(F, t<\zeta | \tau^{-}_0=\infty)=\frac{\underline{\nn}(F, t<\ee_{\beta}<\zeta)}{\underline{\nn}(\ee_{\beta}<\zeta)}=\underline{\nn}(Fg^{+}_{\beta}(X_{t})e^{-\beta t} , t<\zeta).
\end{split}
\end{equation*}
By the monotone convergence theorem, { we also have}  
$$1=\lim_{\beta\to 0}\underline{\nn}\left(g^{+}_{\beta}(X_{t})e^{-\beta t}, t<\zeta\right)=\underline{\nn}\left(g^{+}(X_{t}), t<\zeta\right).$$
An application of Fatou's Fatou's lemma yields $$\underline{\nn}\left(Fg^{+}(X_{t}), t<\zeta\right)\leq \liminf_{\beta\to 0}\underline{\nn}\left(Fg^{+}_{\beta}(X_{t})e^{-\beta t}, t<\zeta\right),$$ 
$$\underline{\nn}\left((1-F)g^{+}(X_{t}), t<\zeta\right)\leq \liminf_{\beta\to 0}\underline{\nn}\left((1-F)g^{+}_{\beta}(X_{t})e^{-\beta t}, t<\zeta\right).$$ From the three facts above we derive the following inequalities 
\begin{equation*}
\begin{split}
\underline{\nn}\left(Fg^{+}(X_{t}), t<\zeta\right)&\leq \liminf_{\beta\to 0}\underline{\nn}\left(Fg^{+}_{\beta}(X_{t})e^{-\beta t}, t<\zeta\right)\\
&\leq \limsup_{\beta\to0} \underline{\nn}\left(Fg^{+}_{\beta}(X_{t})e^{-\beta t}, t<\zeta\right)\leq \underline{\nn}\left(Fg^{+}(X_{t}), t<\zeta\right),
\end{split}
\end{equation*}
\end{proof}

\section{Some calculations under the excursion measure}\label{exc-calcul}
In this section we will first prove Lemma~\ref{Lemma:intensities} and then Theorem~\ref{thm3}.

\begin{proof}[Proof of Lemma~\ref{Lemma:intensities}]
The proof of (i) follows from the well known identity $\displaystyle \nn(\zeta>\mathbf{e}_{\beta})=\frac{1}{u_{\beta}(0)}.$ To prove (ii)  we observe that by Fubini's theorem and the Markov property under the excursion measure, we have
\begin{equation*}
\begin{split}
\nn(\ee_{\beta}<\zeta=\tau^{-}_0<\infty)&=\nn\left((1-e^{-\beta \zeta})\mathbf{1}_{\{\inf_{0<s<\zeta}X_{s}=0,\ X_{\zeta-}=0\}}\right)\\
&=\beta\int^{\infty}_{0}e^{-\beta s}\nn\left(\mathbf{1}_{\{s<\zeta, \inf_{0<s<\zeta}X_{s}=0,\ X_{\zeta-}=0,\ \zeta<\infty\}}\right)\ud s\\
&=\int^{\infty}_{0}\beta e^{-\beta s}\nn\left(\mathbf{1}_{\{s<\zeta,\ X_{s}>0\}}\p_{X_{s}}\left(\tau^{-}_0<\infty, X_{\tau^{-}_0}=0\right)\right)\ud s.
\end{split}
\end{equation*}
The result follows from Kesten's identity,
 $$\p_{x}\left(\tau^{-}_0<\infty, X_{\tau^{-}_0}=0\right)=\frac{\sigma^{2}}{2}\left({W}^{\prime}(x)-\Phi(0)W(x)\right), \qquad x>0,$$
 {(see e.g. identity (44) in \cite{KKR}) and identity (\ref{potentialkilled}).} Indeed, we have that
\begin{equation*}
\begin{split}
\nn(\ee_{\beta}<\zeta=\tau^{-}_0<\infty)&=\frac{\sigma^{2}\beta}{2}\int^{\infty}_{0}e^{-\Phi(\beta)y}\left(W^{\prime}(y)-\Phi(0)W(y)\right)\ud y \\
&=\frac{\sigma^{2}\beta}{2}\left[\frac{\Phi(\beta)}{\Psi(\Phi(\beta))}-\frac{\Phi(0)}{\Psi(\Phi(\beta))}\right]=\frac{\sigma^{2}}{2}\left(\Phi(\beta)-\Phi(0)\right),
\end{split}
\end{equation*}
{where we have used the Laplace transform of scale functions and the definition of  $\Phi.$} To prove the second identity in (ii), observe {from the Markov property that} $\nn(\tau^{-}_{0}=\infty)$ can be strictly positive {whenever}  $X$ drifts towards $\infty,$ that is $0<\Psi^{\prime}(0+)<\infty.$ {On the other hand,}
$$\nn(\tau^{-}_{0}=\infty)=\lim_{\beta\to 0}\nn(\ee_{\beta}<\zeta=\tau^{-}_{0}),$$ and by the Markov property under $\nn$, we get
$$\nn(\tau^{-}_{0}=\infty)=\lim_{\beta\to 0}\nn(\ee_{\beta}<\zeta=\tau^{-}_{0})=\lim_{\beta\to 0}\beta \nn\left(\int^{\zeta}_{0}e^{-\beta s}\mathbf{1}_{\{X_{s}>0\}}\p_{X_{s}}(\tau^{-}_{0}=\infty)\ud s\right).$$ Using the fact that 
$\p_{x}(\tau^{-}_{0}=\infty)=\Psi^{\prime}(0+)W(x),$ $x>0,$ the identity (\ref{potentialkilled}) and the definition of the scale function, { we deduce that the right-hand side in the above display can be written as follows}
$$\lim_{\beta\to 0}\beta\Psi^{\prime}(0+)\int^{\infty}_{0}e^{-\Phi(\beta)x}W(x)\ud x=\lim_{\beta\to 0}\frac{\beta\Psi^{\prime}(0+)}{\Psi(\Phi(\beta))}=\Psi^{\prime}(0+),$$ 
which finishes the proof of (ii). 

{Next we prove (iii). By an application of the Markov property under $\nn$, the compensation formula (\ref{compensation}) and  identity (\ref{potentialkilled}), we deduce}\begin{equation*}
\begin{split}
\nn(0<\tau^{-}_0<\ee_{q}<\zeta)&=\nn\left(e^{-q \tau^{-}_0}\e_{X_{\tau^{-}_0}}\left(1-e^{-q \tau^{+}_{0}}\right)\mathbf{1}_{\{0<\tau^{-}_0<\zeta\}}\right)\\
&=\nn\left(\sum_{0<s<\infty}e^{-qs}\mathbf{1}_{\{I_s>0,\ X_{s}<0,\ { s<\zeta}\}}\e_{X_{s}}\left(1-e^{-q \tau^{+}_{0}}\right)\right)\\
&=\nn\left(\int^{\zeta}_{0}\ud s e^{-qs}\mathbf{1}_{\{X_{s-}>0\}}\int^{0}_{-\infty}\Pi(\ud y )\mathbf{1}_{\{X_{s-}+y<0\}}\e_{X_{s-}+y}\left(1-e^{-q \tau^{+}_{0}}\right)\right)\\
&=\int^{\infty}_{0}\ud z e^{-\Phi(q)z}\int^{0}_{-\infty}\Pi(\ud y)\mathbf{1}_{\{z+y<0\}}\left(1-e^{(z+y)\Phi(q)}\right)\\
&=\Phi(q)\int^{\infty}_{0}\ud z{\overline{\Pi}^{-}(z)}ze^{-z\Phi(q)}.
\end{split}
\end{equation*}
The second assertion in (iii) follows by making $q\to 0.$

We will next see that (iv) can be established as {in }(ii) using an argument of duality, but first we recall some important facts. It is well known that $X$ and the process $\widehat{X}:=-X$ are in weak duality with respect to Lebesgue measure, and because the resolvent of $X$ is also absolutely continuous with respect to Lebesgue measure, the  process $X$ and $\widehat{X}$ are in classical duality. The process $\widehat{X}$ is a spectrally positive L\'evy process, for which $0$ is a regular for itself, its resolvent has densities $\widehat{u}_{q}$, {satisfying} $\widehat{u}_{q}(x)={u}_{q}(-x), x\in\mathbb{R},$ $q>0,$ and we denote by $\widehat{\nn}$ its excursion measure {away} from zero. We denote by $\rho$ the time reversal mapping of the excursion at its lifetime $\zeta$, i.e
$$\rho X _{t}=\begin{cases}
X_{(\zeta-t)-}, & 0<t<\zeta<\infty\\
\Delta, & \text{otherwise}.\end{cases}$$ Observe that {since} $0$ is regular for itself,  the excursions {away }from zero for $X$ start and end at zero. It follows from the results in Section 4 of \cite{getoor-sharpe} that the image of $\widehat{\nn}$ under the mapping $\rho,$ say $\rho\widehat{\nn},$ is ${\nn}.$ This implies the equality
$${\nn}(\tau^{-}_0=0, \ee_{\beta}<\zeta<\infty)=\widehat{\nn}(\infty>\tau^{+}_{0}=\zeta>\ee_{\beta})=\widehat{\nn}(X_{\tau^{+}_{0}}=0, X_{\ee_{\beta}}<0, \ee_{\beta}<\zeta<\infty).$$ Proceeding as in the proof of (ii), we get 
\begin{equation*}
\begin{split}
\beta\int^{\infty}_{0}\ud t e^{-\beta t} &\int_{y\in (-\infty,0)}\widehat{\nn}(X_{t}\in \ud y, t<\zeta)\widehat{\p}_{y}\left(X_{\tau^{+}_{0}}=0, \tau^{+}_{0}<\infty\right)\\
&=\beta\int^{\infty}_{0}\ud t e^{-\beta t} \int_{y\in (-\infty,0)}\widehat{\nn}(X_{t}\in \ud y, t<\zeta){\p}_{-y}\left(X_{\tau^{-}_0}=0,  \tau_0^-<\infty\right)\\
&=\beta \int^{0}_{-\infty}\frac{\widehat{u}_{\beta}(y)}{\widehat{u}_{\beta}(0)}{\p}_{-y}\left(X_{\tau^{-}_0}=0, \tau^{-}_{0}<\infty\right)\ud y \\
&=\beta \int^{0}_{-\infty}\frac{{u}_{\beta}(-y)}{{u}_{\beta}(0)}{\p}_{-y}\left(X_{\tau^{-}_0}=0, \tau^{-}_{0}<\infty\right)\ud y \\
&=\frac{\sigma^{2}}{2}\left(\Phi(\beta)-\Phi(0)\right).
\end{split}
\end{equation*}
The second identity in (iv) {follows from (i) and (iii), thanks to the identity}
$$\nn(0=\tau^{-}_{0}<\zeta=\infty)=\nn(\zeta=\infty)-\nn(0<\tau^{-}_{0}<\zeta=\infty)=\nn(\zeta=\infty)-\lim_{\beta\to 0}\nn(0<\tau^{-}_{0}<\ee_{\beta}<\zeta),$$ and the fact that 
when $X$ drifts towards $-\infty$, 
$$\nn(\zeta=\infty)=\Psi^{\prime}(\Phi(0)+)=\widehat{\kappa}(\Phi(0))=\frac{\sigma^{2}\Phi(0)}{2}+\Phi(0)\int^{\infty}_{0}e^{-\Phi(0)y}y\overline{\Pi}^{-}(y)\ud y.$$ 
The identity in (v) is deduced from  identities (i) to (iv), using the following decomposition
\begin{equation*}
\begin{split}
\nn(\ee_{\beta}<\tau^{-}_{0}<\zeta)=\nn(\zeta>\ee_{\beta})-\left(\nn(\tau^{-}_{0}=0, \ee_{\beta}<\zeta)+\nn(0<\tau^{-}_{0}<\ee_{\beta}<\zeta)+\nn(\ee_{\beta}<\tau^{-}_{0}=\zeta)\right),
\end{split}
\end{equation*}
and elementary but tedious calculations. 
\end{proof}

\section{Proof of Theorem~\ref{thm3}}\label{Proofthm3}
\begin{proof}[Proof of Theorem~\ref{thm3}] {We first prove part (i). For any $n\geq 1,$ $t_{1}<\cdots<t_{n}<\infty$ and  $f:(-\infty,0)^{n}\to\mathbb{R}^{+}$ a measurable and continuous function  with compact support, we have}
\begin{equation*}
\begin{split}\label{cem}
\nn\Big(f(X_{t_{1}},&\ldots, X_{t_{n}}), t_{n}<\tau^{+}_{0}, t_{n}<\zeta\wedge\mathbf{e}_{\beta}\Big)\\
&=\lim_{t\to0}\nn\left(f(X_{t_{1}+t},\ldots, X_{t_{n}+t}), t+t_{n}<\tau^{+}_{0}, t+t_{n}<\zeta\wedge\mathbf{e}_{\beta}\right)\\
&=\lim_{t\to0}\nn\left(\e_{X_{t}}\left( f(X_{t_{1}},\ldots, X_{t_{n}}), t_{n}<T_{0}\wedge \mathbf{e}_{\beta}\right), t<\tau^{+}_{0}, t<\zeta\wedge \mathbf{e}_{\beta}\right)\\
&=\lim_{t\to0}\nn\left(\e^{\updownarrow,\beta}_{X_{t}}\left( f(X_{t_{1}},\ldots, X_{t_{n}})\frac{1}{h_{\beta}(X_{t_{n}})}, t_{n}<\zeta\right){h_{\beta}(X_{t})} , t<\tau^{+}_{0}, t<\zeta\wedge \mathbf{e}_{\beta}\right)\\
&=\e^{\updownarrow,\beta}_{0-}\left(f(X_{t_{1}},\ldots, X_{t_{n}})\frac{\Phi^{\prime}(\beta)\Phi(\beta)}{h_{\beta}(X_{t_{n}})}\right)\lim_{t\to 0} \frac{\nn\left({h_{\beta}(X_{t})} , t<\tau^{+}_{0}, t<\zeta\wedge \mathbf{e}_{\beta}\right)}{\Phi^{\prime}(\beta)\Phi(\beta)},
\end{split}
\end{equation*} 
{where in the first identity, we used that  excursions start from $0$ continuously, which is a consequence of the regularity of zero; then the second identity follows from  the absence of positive jumps; the third and fourth identities follow from an application of the Markov property under $\nn$ and the construction of the measure $\e^{\updownarrow,\beta}_{\cdot}$.   
From the Markov property and the definition of $h_{\beta}$ in (\ref{hb}), the second factor is given by 
\begin{equation*}
\begin{split}
\lim_{t\to 0} \frac{\nn\left({h_{\beta}(X_{t})} , X_{t}<0, t<\zeta\wedge \mathbf{e}_{\beta}\right)}{\Phi^{\prime}(\beta)\Phi(\beta)}=\frac{\nn\left(\tau^{-}_{0}=0<\mathbf{e}_{\beta}<\zeta\right)}{\Phi(\beta)}.
\end{split}
\end{equation*}
According to Lemma~\ref{Lemma:intensities}, part (iv), 
 the right most term in the above expression equals $\sigma^{2}/2$. Furthermore, Theorem~\ref{teo:1}, part (i), ensures that the first factor in the right-hand side  of (\ref{cem})  equals 
$$\widehat{\overline{\nn}}(f(X_{t_{1}},\ldots, X_{t_{n}}), t<\zeta\wedge \mathbf{e}_{\beta}).$$ 
The result follows since the above proved facts imply the equality of the finite dimensional distributions under $\nn\mathbf{1}_{\{\tau^{-}_{0}=0, \ee_{\beta}<\zeta\}}$.}
{We now prove part (ii). Let $n\geq 1,$ $t_{1}<\cdots<t_{n}<\infty,$ $f:(0,\infty)^{n}\to\mathbb{R}^{+}$  be a measurable and  continuous function with compact support on $(0,\infty)^{n}$. Arguing as in the proof of part (i),} we obtain 
\begin{equation*}
\begin{split}
\nn&\left(f(X_{t_{1}},\ldots, X_{t_{n}}), X_{\tau^{-}_{0}}=0, \tau^{-}_{0}<\infty, t_{n}<\mathbf{e}_{\beta}<\zeta\right)\\
&={\nn\left(f(X_{t_{1}},\ldots, X_{t_{n}})\p_{X_{t_{n}}}(\ee_{\beta}<\tau^{-}_{0}<\infty, X_{\tau^{-}_{0}}=0), t_{n}<\tau^{-}_{0}, t_{n}<\zeta\wedge \ee_{\beta}\right)}\\
&=\lim_{t\to0}\nn\left(f(X_{t_{1}+t},\ldots, X_{t_{n}+t})\p_{X_{t_{n}+t}}(\ee_{\beta}<\tau^{-}_{0}<\infty, X_{\tau^{-}_{0}}=0), t+t_{n}<\tau^{-}_{0}, t+t_{n}<\zeta\wedge\mathbf{e}_{\beta}\right)\\
&=\lim_{t\to0}\nn\left(h_{\beta}(X_{t}), t<\tau^{-}_{0}, t<\zeta\right)\\
&\qquad \times \lim_{x\to 0+}\frac{1}{h_{\beta}(x)}\e_{x}\left(f(X_{t_{1}},\ldots, X_{t_{n}})\p_{X_{t_{n}}}(\ee_{\beta}<\tau^{-}_{0}<\infty, X_{\tau^{-}_{0}}=0), t_{n}<\tau^{-}_{0}\wedge \mathbf{e}_{\beta}\wedge \zeta\right).
\end{split}
\end{equation*}
According to Lemma~\ref{Lemma:intensities}, we have 
$$\lim_{t\to 0+}\nn(h_{\beta}(X_{t}),t<\tau^{-}_{0}, t<\zeta)=\Phi^{\prime}(\beta)\nn(\ee_{\beta}<\zeta, X_{0+}> 0)=1-\frac{\sigma^{2}}{2}\Phi^{\prime}(\beta)\Phi(\beta).$$ 
{Then arguing as in  Lemma \ref{Lemma3} and using the limits from Lemma~\ref{lemma:h0},}
we deduce that the second limit in the equation above is equal to 
$$\frac{1}{1-\frac{\sigma^{2}}{2}\Phi^{\prime}(\beta)\Phi(\beta)}\underline{\nn}\left(f(X_{t_{1}},\ldots, X_{t_{n}})\p_{X_{t_{n}}}(\ee_{\beta}<\tau^{-}_{0}<\infty, X_{\tau^{-}_{0}}=0), t_{n}<\zeta\wedge \ee_{\beta}\right).$$ 
The Markov property allows to conclude
that the following identity holds
\begin{equation*}
\begin{split}
\nn\Big(f(X_{t_{1}},\ldots, X_{t_{n}}), X_{\tau^{-}_{0}}=0, \tau^{-}_{0}<\infty,& t_{n}<\mathbf{e}_{\beta}<\zeta\Big)\\
&=\underline{\nn}\left(f(X_{t_{1}},\ldots, X_{t_{n}}), X_{\zeta}=0, t_{n}< \ee_{\beta}<\zeta<\infty\right).
\end{split}
\end{equation*}
This concludes the proof of part (ii), since the proved equality determine the finite dimensional distributions under $\nn|_{\{\ee_{\beta}<\tau^{-}_{0}=\zeta\}}.$

{Finally, we prove part (iii).} Let $t>0,s\geq 0$ $F\in\mathcal{C}_{t}$ and $G\in\mathcal{C}_{s}.$ We have by the Markov property under $\nn$ that
{\begin{equation*}
\begin{split}
\nn\Big(F\mathbf{1}_{\{t<\tau^{-}_{0}, X_{\tau^{-}_{0}}<0\}}& f(X_{\tau^{-}_{0}-}, X_{\tau^{-}_{0}})G\circ{\theta_{\tau_{0}^-}}, {\tau_{0}^-}+s<T_{0}\Big)\\
&=\nn\left({F}\mathbf{1}_{\{t<\tau^{-}_{0}, X_{\tau^{-}_{0}}<0\}}f(X_{\tau^{-}_{0}-}, X_{\tau^{-}_{0}})\e^{0}_{X_{\tau^{-}_{0}}}\left(G, s<T_{0}\right)\right)\\
&=\nn\left({F}\mathbf{1}_{\{t<\tau^{-}_{0}, X_{\tau^{-}_{0}}<0\}}f(X_{\tau^{-}_{0}-}, X_{\tau^{-}_{0}})H(X_{\tau^{-}_{0}})\right), 
\end{split}
\end{equation*}
where $H(z)=\e^{0}_{z}\left(G, s<T_{0}\right)$, for  $z<0.$ }From the compensation formula (\ref{compensation}), we deduce that the latter can be written as 
\begin{equation*}
\begin{split}
&\nn\left(\int^{\infty}_{t}\ud u {F}\mathbf{1}_{\{t<\tau^{-}_{0}, X_{u}>0{, u<\zeta}\}}\int^{0}_{-\infty}\Pi(\ud z) f(X_{u}, X_{u}+z)H(X_{u}+z)\mathbf{1}_{\{X_{u}+z<0\}}\right).
\end{split}
\end{equation*}
{Applying the Markov property again but now at time $t$, we deduce}
\begin{equation*}
\begin{split}
&\nn\left({F}\mathbf{1}_{\{t<\tau^{-}_{0}, X_{t}>0\}}\e_{X_{t}}\left(\int^{T_{0}}_{0}\ud s\mathbf{1}_{\{X_{s}>0\}}\int^{0}_{-\infty}\Pi(\ud z) f(X_{s}, X_{s}+z)H(X_{s}+z)\mathbf{1}_{\{X_{s}+z<0\}}\right)\right).
\end{split}
\end{equation*}
The above {identity} can be written in terms of the $0$-potential of the process $X$ killed at its first hitting time of $0$, $U^{0}_{0},$ i.e.
\begin{equation*}
\begin{split}
&\nn\left({F}\mathbf{1}_{\{t<\tau^{-}_{0}, X_{t}>0\}}\int_{(0,\infty)}U^{0}_{0}(X_{t},\ud y)\int^{0}_{-\infty}\Pi(\ud z) f(y, y+z)H(y+z)\mathbf{1}_{\{y+z<0\}}\right).
\end{split}
\end{equation*}
The absence of positive jumps implies that the restriction to $(0,\infty),$ of the potential the process killed at its first hitting time of $0,$ should be equal to that of the process killed at its first passage time below $0,$ that is
$$\e_{x}\left(\int^{T_{0}}_{0}1_{\{X_{t}\in\ud y\}}\ud t\right){\mathbf{1}_{\{y>0\}}}=\e_{x}\left(\int^{\tau^{-}_{0}}_{0}\mathbf{1}_{\{X_{t}\in \ud y\}}\ud t\right)\mathbf{1}_{\{y>0\}},$$ 
This observation together with Theorem 2.7 in page 123 in \cite{KKR} gives as in (\ref{eq:trrrr}) the identity
\begin{equation*}
\begin{split}
&\int^{\infty}_{0}U^{0}(x,\ud y)\int^{0}_{-\infty}\Pi(\ud z) f(y, y+z)H(y+z)\mathbf{1}_{\{y+z<0\}}\\
&=\int^{\infty}_{0}\ud y \left(e^{-\Phi(0)y}W(x)-W(x-y)\right)\int^{0}_{-\infty}\Pi(\ud z) f(y, y+z)H(y+z)\mathbf{1}_{\{y+z<0\}}=KfH(x),\\
\end{split}
\end{equation*}
with $x>0,$ and where $K$ is the kernel defined in (\ref{KB}). In order to conclude the proof of part (iii), it just remains to describe the term $\nn\left(F\mathbf{1}_{\{t<\tau^{-}_{0}<\zeta\}}\right).$ But arguing as in the proof of part (ii),  we obtain, for $t>0,$ and $F\in\mathcal{H}_{b,t}$ the following equality 
\begin{equation*}
\begin{split}
&\nn\left({F}\mathbf{1}_{\{t<\tau^{-}_{0}<\zeta\wedge\ee_{\beta}\}}\right)=\underline{\nn}\left(F\mathbf{1}_{\{X_{\zeta-}>0\}},  t<\zeta< \ee_{\beta}\right).
\end{split}
\end{equation*}This finishes the proof.\end{proof}

\noindent{\bf Acknowledgements} 
VR acknowledges support from CONACyT grant number 234644 and EPSRC grant number EP/M001784/1. This work was concluded whilst VR was on sabbatical at the University of Bath, he gratefully acknowledges the kind hospitality of the Department and University.


\end{document}